\documentclass[12pt]{amsart}
\usepackage{amsmath}
\allowdisplaybreaks
\usepackage{amssymb}
\usepackage{verbatim}

\newtheorem{theorem}{Theorem}[section]
\newtheorem{lemma}[theorem]{Lemma}
\newtheorem{proposition}[theorem]{Proposition}

 \theoremstyle{definition}
\newtheorem{definition}[theorem]{Definition}

\theoremstyle{remark}
\newtheorem{remark}[theorem]{Remark}

\numberwithin{equation}{section}

\begin{document}

\title[$(\mathbf{p}, k)$-convex hypersurfaces for prescribed curvature]{The existence of $(\mathbf{p}, k)$-convex hypersurfaces for a class of Hessian quotient type curvature equations}

\author{Jiabao Gong}
\address{Faculty of Mathematics and Statistics, Hubei Key Laboratory of Applied Mathematics, Hubei University,  Wuhan 430062, P.R. China}
\email{202321104011284@stu.hubu.edu.cn}

\author{Qiang Tu$^{\ast}$}
\address{Faculty of Mathematics and Statistics, Hubei Key Laboratory of Applied Mathematics, Hubei University,  Wuhan 430062, P.R. China}
\email{qiangtu@hubu.edu.cn}

\thanks{$\ast$ Corresponding author}


\date{}

\begin{abstract}
This article investigates the existence of closed, star-shaped hypersurfaces for a class of Hessian quotient type curvature equations, in which the operator $\frac{\sigma_k}{\sigma_l}(\Lambda)$ arising in these equations can be viewed as a generalization of the classical Hessian quotient operator. By combining a priori estimates with the continuity method, we establish the existence and uniqueness of $(\mathbf{p}, k)$-convex hypersurfaces for both nonhomogeneous and homogeneous equations of this type. Furthermore, by exploiting the recently discovered ``inverse convexity'' property of the operator $\frac{\sigma_k}{\sigma_l}(\Lambda)$, we prove a constant rank theorem and thereby obtain the existence and uniqueness of strictly convex solutions  to these curvature equations.
\end{abstract}


\makeatletter
\@namedef{subjclassname@2020}{\textup{2020} Mathematics Subject Classification}
\makeatother

\maketitle
\vskip4ex

\section{Introduction}

Let $M\subset \mathbb{R}^{n+1}$ be a hypersurface and $\kappa(X)=(\kappa_1, \cdots, \kappa_n)$ be the principal curvatures of $M$ at $X$. Given an integer $\mathbf{p}$ with $1\leq \mathbf{p}\leq n$, set
$$\mathfrak{J}:=\{(i_1,\cdots,i_\mathbf{p})|1\leq i_1<\cdots<i_\mathbf{p}\leq n\}.$$
For any $I=(i_1, \cdots, i_\mathbf{p}) \in \mathfrak{J}$, we define
$$\Lambda_I(\kappa)=\kappa_{i_1}+\cdots+\kappa_{i_\mathbf{p}}.$$
For convenience, we  fix an order of the elements in $\mathfrak{J}$:
\begin{eqnarray*}
I_1,\cdots,I_N,~~\mbox{with}~~N=C^\mathbf{p}_n=\frac{n!}{\mathbf{p}!(n-\mathbf{p})!}.
\end{eqnarray*}
For an integer $1\leq k \leq N$, we give the following definition.
\begin{definition}
A  regular $C^2$-hypersurface $M$ is called $(\mathbf{p}, k)$-convex if, at each $X\in M$, $\kappa(X)$ satisfies  $\Lambda(\kappa) \in \tilde{\Gamma}_{k}$, where
$\Lambda(\kappa)=(\Lambda_{I_1}(\kappa), \cdots, \Lambda_{I_N}(\kappa))$,
and the G\aa rding's cone $\tilde{\Gamma}_{k}$ is defined as
\begin{eqnarray*}
\tilde{\Gamma}_{k}=\{\lambda \in \mathbb{R}^N: \sigma_{j}(\lambda)>0,~ \forall ~ 1\leq j \leq k\}.
\end{eqnarray*}
\end{definition}
When $k=N$, $(\mathbf{p}, k)$-convex hpersurface implies that $\kappa_1+\cdots+\kappa_{\mathbf{p}}>0$ where $\kappa_1\leq \kappa_2 \leq \cdots \leq\kappa_{n}$. This is equivalent to the definition of $\mathbf{p}$-convexity, a concept originally introduced by Wu \cite{W-1987}, and subsequently studied in depth by Sha \cite{Sha-86, Sha-87} and Harvey-Lawson \cite{Har-12, Har-13}. Hence, the definition of $(\mathbf{p}, k)$-convex can be regarded as a generalization of the definition of $\mathbf{p}$-convex.

The curvature equation arises naturally in many geometric problems such as prescribed Weingarten curvature problem and prescribed curvature measure problem in convex geometry (see \cite{Be69, Chu21, Fi67, Fi68, Ger06, HMS04}).
In this paper, we consider the following Hessian quotient type curvature equation for $(\mathbf{p}, k)$-convex hypersurface
\begin{equation}\label{G-eq-0000}
\frac{\sigma_k}{\sigma_l}(\Lambda(\kappa))=f(X, \nu), \quad X\in M,
\end{equation}
where $\sigma_k$ is the $k$-th elementary symmetric function, $0\leq l <k\leq  N$, $\nu(X)$ is the unit outer normal of $M$ at $X$ and  the function $f \in C^2$ is positive. This type of Hessian quotient type curvature equation has been extensively studied in some special cases.

When $\mathbf{p}=1$ and $l=0$, Eq. \eqref{G-eq-0000} reduces to the classical prescribed curvature  equation
$$\sigma_k(\kappa)=f(X, \nu),$$
which has been widely studied in the past three decades.   When $k=1$ or  $k=n$, the equation becomes the quasilinear equation or the Gauss curvature equation, respectively. Then the corresponding  a priori estimates and solvability follow from the classical theory of quasi-linear PDEs and Monge-Amp\`ere type equations in  \cite{CNS}.
When $f$ is independent of $\nu$, curvature estimates for admissible hypersurfaces were proved by Caffarelli-Nirenberg-Spruck \cite{Ca}
for a general class of fully nonlinear operators $F$, including
$F=\sigma_k$ and $F=\frac{\sigma_k}{\sigma_l}$.
Then Guan-Lin-Ma \cite{GM06} studied the existence and convexity of shar-shaped hypersurfaces  by the Constant Rank Theorem.
 When $f$ depends only on $\nu$, Guan-Guan \cite{Guan02} used a priori estimates and degree theory approach to prove the existence of a closed, strictly convex hypersurface satisfying the curvature equation.
 When  $f(X, \nu) = \langle X, \nu \rangle \Tilde{f}(X)$, the equation is  equivalent to the curvature measure problem, which was settled by Guan-Li-Li \cite{Guan12}, Guan-Lin-Ma \cite{GM06}, Guan-Li \cite{Guan-Li-97},  through the establishment of a crucial a priori $C^2$ esimate and a deformation lemma.
 Then Huang-Xu \cite{HX-2013} studied  the generalized curvature measure problem with  $f(X, \nu) = \langle X, \nu \rangle^p \tilde{f}(X)$.
Ivochkina \cite{Iv1,Iv2} considered the Dirichlet problem of Eq. \eqref{G-eq-0000} and obtained curvatute estimates under some extra conditions on the dependence of $f$ on $\nu$. Moreover, there has been much progress in establishing curvature estimates for Eq. \eqref{G-eq-0000} in case $2\leq k \leq n-1$.  We refer to \cite{Chu21, Guan-Ren15, Lu-2023, Ren, Ren1,  Sp}.

When $\mathbf{p}=n-1$, Eq. \eqref{G-eq-0000} becomes the following prescribed curvature type problem
\begin{equation}\label{G-eq-000011}
\frac{\sigma_k}{\sigma_l}(\eta(\kappa))=f(X, \nu),
\end{equation}
where $\eta= (\sum_{i\neq 1} \kappa_i, \cdots, \sum_{i\neq n} \kappa_i)$. This kind of equation is motivated by many important geometric problems. For example, when $k=n$, $l=0$ and $\eta$ is replaced by the Hermitian matrix $\Delta u I- \partial \bar{\partial} u$, Eq. \eqref{G-eq-000011}  is called the $(n-1)$-Monge-Amp\`ere equation, which is related to the Gauduchon conjecture in complex geometry; the conjecture has recently been solved by Sz\'ekelyhidi-Tosatti-Weinkove in \cite{STW-17}. For more details, see \cite{G-84, TW-17}. It was then natural to study a  priori estimates and solvability
for Eq. \eqref{G-eq-000011}.  Chu-Jiao \cite{CJ-2021} and Chen-Tu-Xiang \cite{CTX2020} established curvature estimates for Eq. \eqref{G-eq-000011}. The Plateau type problem for strictly locally convex hypersurfaces with Eq. \eqref{G-eq-000011}
has been studied by He-Tu-Xiang in \cite{HTX-2024}. The readers can refer to \cite{CTX-2023, DW-2022, Mei-2023, MZ24, SX-2025, Wang-2025} for more researches about Eq. \eqref{G-eq-000011}.

Recently, curvature estimates and existence of a star-shaped $\mathbf{p}$-convex hypersurface for Eq. \eqref{G-eq-0000}  with $k=n, l=0$ have been obtained by Dong \cite{Dong-2023}. For  general $k$ and $l$,  Zhou \cite{Zhou-2024} established the global curvature estimates under some conditions. A natural problem is raised whether we can establish   existence results for closed star-shaped hypersurfaces for this curvature type  equation. In this paper, we try to discuss this problem. More precisely, inspired by the  works of Guan-Li-Li \cite{Guan12} and Guan-Lin-Ma \cite{GM06} on curvature measure problem,  we consider the following prescribed curvature  measure type problem
\begin{equation}\label{G-eq}
\frac{\sigma_k}{\sigma_l}(\Lambda(\kappa))=f(\frac{X}{|X|}) |X|^b \langle X, \nu\rangle^q, \quad X\in M,
\end{equation}
for closed, star-shaped hypersurfaces, where $b, q$ are real numbers.

By establishing the a priori estimates for the $(\mathbf{p}, k)$-convex solution to  equation \eqref{G-eq}, we obtain the following main results.

\begin{theorem}\label{k-convex}
Let $n \geq 2$, $\mathbf{p}>1$, $2<k\leq C_{n-1}^{\mathbf{p}-1}+1$ and $0 \leq l<\min\{C_{n-1}^{\mathbf{p}-1}, k\}  $. Suppose $f \in C^2(\mathbb{S}^n)$ is a positive function.
\begin{enumerate}
    \item[{\em(\romannumeral1)}] If $-b-q-k+l>0$, then there exists a unique  $(\mathbf{p}, k)$-convex star-shaped hypersurface $M \in C^{3,\alpha}$ with $ \alpha \in (0,1)$, satisfying \eqref{G-eq}.
    \item[{\em(\romannumeral2)}] If $-b-q-k+l=0$, and one of the following conditions holds:
\begin{enumerate}
\item  $q \neq 0$;
 \item   $q =0$ and $\frac{|\nabla f|}{f}< 2(k-l) \sqrt{\frac{\mathbf{p}-1}{\mathbf{p}}}$;
 \end{enumerate}
then there exists a unique positive constant $\gamma$ such that  the homogeneous Hessian quotient type curvature equation
    \begin{equation}\label{eq:1.1}
    \frac{\sigma_k}{\sigma_l}(\Lambda(\kappa))=\gamma f(\frac{X}{|X|}) |X|^b \langle X, \nu\rangle^q, \quad X \in M,
    \end{equation}
    has a unique  $(\mathbf{p}, k)$-convex  hypersurface $M \in C^{3,\alpha}$, $\alpha \in (0,1)$, up to a dilation.
\end{enumerate}
Moreover,  whatever the addition in $(\romannumeral1)$ and $(\romannumeral2)$ assume that $q\geq0$ and
\begin{equation}\label{908998}
f^{\frac{1}{k-l}}(\frac{X}{|X|}) |X|^{\frac{b}{k-l}}~\mbox{is locally concave in}~\mathbb{R}^{n+1} \backslash \{0\},
\end{equation}
 then $M$ is strictly convex.
\end{theorem}

\begin{remark}
When applying the continuity method to prove the existence of strictly convex solutions, we need strict convexity to be preserved throughout the process.  For this purpose, we have established the Constant Rank Theorem of Eq. \eqref{G-eq} using the new phenomenon of ``inverse convexity'' of operators $\frac{\sigma_k}{\sigma_l}(\Lambda)$.
Furthermore, when $q = 0$, $b = -(k-l)$, and $\mathbf{p}=n-1$, Eq.~\eqref{eq:1.1} reduces to the prescribed mixed Weingarten curvature equation
\begin{equation*}\label{eq:mixed}
    \frac{\sigma_k}{\sigma_l}(\eta(\kappa))=\gamma f(\frac{X}{|X|}) |X|^{-(k-l)},
\end{equation*}
which has been studied by Mei-Zhu in \cite{MZ24}. Compard to Mei-Zhu's work, Theorem \ref{k-convex} allows more general constraints on $f$ in the homogeneous case.
\end{remark}

It is worth noting that in order to obtain the  curvature estimates, which have already been obtained by Zhou \cite{Zhou-2024}, we have to add constraints on $k$ and $l$ in Theorem  \ref{k-convex}.
However, we expect that the curvature estimates for Eq. \eqref{G-eq} still hold without any restrictions on $k$  in some special cases.
For this reason,  we consider the following Hessian type curvature equation
\begin{equation}\label{G-eqk}
\sigma_k(\Lambda(\kappa))=f(\frac{X}{|X|}) |X|^b \langle X, \nu\rangle^q,
\end{equation}
where $0<k\leq N:=C_n^\mathbf{p}$. The following are our results.

\begin{theorem}\label{k-convexk1}
Let $n \geq 2$, $\mathbf{p}> 1$, $1\leq k\leq N$ and  $q\leq 1$. Suppose $f \in C^2(\mathbb{S}^n)$ is a positive function.
\begin{enumerate}
\item[{\em(\romannumeral1)}]   If $-b-q-k>0$, then there exists a unique  $(\mathbf{p}, k)$-convex star-shaped hypersurface $M \in C^{3,\alpha}$ with $\alpha \in (0,1)$, satisfying \eqref{G-eqk}.
\item [{\em(\romannumeral2)}] If $-b-q-k=0$, and one of the following conditions holds:
\begin{enumerate}
\item  $q \neq 0$;
 \item   $q =0$ and $\frac{|\nabla f|}{f}<2k\sqrt{\frac{\mathbf{p}-1}{\mathbf{p}}}$;
 \end{enumerate}
 then there exists a unique positive constant $\gamma$ such that the homogeneous Hessian type curvature equation
    \begin{equation}\label{eq:1.2}
    {\sigma_k}(\Lambda(\kappa))=\gamma f(\frac{X}{|X|}) |X|^b \langle X, \nu\rangle^q, \quad X \in \mathbb{S}^n,
    \end{equation}
    has a unique  $(\mathbf{p}, k)$-convex star-shaped hypersurface $M \in C^{3,\alpha}$, $\alpha \in (0,1)$, up to a dilation.
\end{enumerate}
Furthermore, whatever the addition in $(\romannumeral1)$ and $(\romannumeral2)$ assume that $q\geq0$ and
$$f^{\frac{1}{k}}(\frac{X}{|X|}) |X|^{\frac{b}{k}}~\mbox{is locally concave in}~\mathbb{R}^{n+1}  \setminus \{0\},$$
then $M$ is strictly convex.
\end{theorem}

\begin{remark}
When $k=n$, Eq. \eqref{k-convexk1}  reduces to the following curvature equation
\begin{equation}\label{eq:123}
\Pi_{1\leq i_1<\cdots<i_{\mathbf{p}}}(\kappa_{i_1}+\cdots+\kappa_{i_{\mathbf{p}}})= f(\frac{X}{|X|}) |X|^b \langle X, \nu\rangle^q.
\end{equation}
Dong \cite{Dong-2023} established curvature estimates for $\mathbf{p}$-convex hypersurfaces for  Eq. \eqref{eq:123} when  $\mathbf{p}\geq\frac{n}{2}$,  with $f(\frac{X}{|X|}) |X|^b \langle X, \nu\rangle^q$ replaced by $f(X,\nu)$. It remains unknown whether the curvature estimates hold  when $\mathbf{p}<\frac{n}{2}$.  Theorem \ref{k-convexk1} not only provides  curvature estimates for  Eq. \eqref{eq:123} when $\mathbf{p}>1$, but also presents the existence results for $\mathbf{p}$-convex hypersurfaces and strictly convex hypersurfaces.
\end{remark}

This article is organized as follows. In Section 2, we introduce the necessary preliminaries, including the $(\mathbf{p},k)$-cone and $(\mathbf{p},k)$-convex function.
In particular, we prove the ``inverse convexity'' of the operator $-\left[\frac{\sigma_{k}}{\sigma_{l}}(\Lambda)\right]^{-\frac{1}{k-l}}$,
which plays a crucial role in the proof of the Constant Rank Theorem.
In Section 3, we derive the  a priori estimates and establish a Constant Rank Theorem, which guarantees the convexity of solutions and yields the existence and uniqueness result for the case $-b-q-k+l>0$. In Section 4, we establish the corresponding existence and uniqueness result for the case  $-b-q-k+l=0$.

\section{Preliminaries}

\subsection{Setting and General facts}
For later convenience, we first state our conventions on the Riemannian
curvature tensor and derivative notation. Let $M$ be a smooth
manifold and $g$ be a Riemannian metric on $M$ with Levi-Civita
connection $\nabla$. For a $(s, r)$-tensor field $\alpha$ on $M$,
its covariant derivative $\nabla \alpha$ is a $(s, r+1)$-tensor
field given by
\begin{equation*}
\begin{aligned}
&\nabla \alpha(Y^1, .., Y^s, X_1, ..., X_r, X)\\
=&\nabla_{X} \alpha(Y^1, .., Y^s, X_1, ..., X_r)\\
=&X(\alpha(Y^1, .., Y^s, X_1, ..., X_r))-\alpha(\nabla_X Y^1, .., Y^s, X_1, ..., X_r)\\
&-...-\alpha(Y^1, ..,Y^s, X_1, ..., \nabla_X  X_r),
\end{aligned}
\end{equation*}
the coordinate expression of this is denoted by
$$\nabla \alpha=(\alpha_{k_{1}\cdot\cdot\cdot
k_{r}; k_{r+1}}^{l_{1}\cdot\cdot\cdot l_{s}}).$$
We next  define the second covariant derivative of $\alpha$ as follows:
\begin{equation*}
\begin{aligned}
&&\nabla^2 \alpha(Y^1, .., Y^s, X_1, ..., X_r, X, Y)
=(\nabla_{Y}(\nabla\alpha))(Y^1, .., Y^s, X_1, ..., X_r, X).
\end{aligned}
\end{equation*}
The coordinate expression of this is denoted by
$$\nabla^2 \alpha=(\alpha_{k_{1}\cdot\cdot\cdot
k_{r}; k_{r+1}k_{r+2}}^{l_{1}\cdot\cdot\cdot l_{s}}).$$
 Similarly, we can also define the higher order covariant derivative of $\alpha$:
$$\nabla^3 \alpha=\nabla(\nabla^2 \alpha), \nabla^4 \alpha=\nabla(\nabla^3 \alpha), ... ,$$
and so on. For simplicity, the coordinate expression of the
covariant differentiation will usually be denoted by indices without
semicolons, e.g.,
$$u_{i}, \quad u_{ij} \quad \mbox{or} \quad
u_{ijk}$$
 for a function $u: M\rightarrow \mathbb{R}$.

Our convention for the Riemannian curvature $(3,1)$-tensor $Rm$ is
\begin{equation*}
Rm(X, Y)Z=-\nabla_{X}\nabla_{Y}Z+\nabla_{Y}\nabla_{X}Z+\nabla_{[X,Y]}Z.
\end{equation*}
Choose a local coordinate chart $\{x^i\}_{i=1}^{n}$ of $M$. The
components of the $(3,1)$-tensor $Rm$ are defined by
\begin{equation*}
Rm\bigg({\frac{\partial}{\partial x^i}}, {\frac{\partial}{\partial
x^j}}\bigg){\frac{\partial}{\partial x^k}}=R_{ijk}^{\ \ \
l}{\frac{\partial}{\partial x^l}}
\end{equation*}
and $R_{ijkl}=g_{lm}R_{ijk}^{\ \ \ m}$. Then, we have the
standard commutation formulas (Ricci identities):
\begin{equation*}\label{RI}
\alpha_{k_{1}\cdot\cdot\cdot k_{r};\ j i}^{l_{1}\cdot\cdot\cdot
l_{s}}-\alpha_{k_{1}\cdot\cdot\cdot k_{r};\ i
j}^{l_{1}\cdot\cdot\cdot l_{s}}=\sum_{a=1}^{r}R^{\ \ \ m}_{ijk_{l}}
\alpha_{k_{1}\cdot\cdot\cdot k_{a-1}m k_{a+1}\cdot\cdot\cdot
k_{r}}^{l_{1}\cdot\cdot\cdot l_{s}}-\sum_{b=1}^{s}R^{\ \ \
l_b}_{ijm} \alpha_{k_{1}\cdot\cdot\cdot k_{r}}^{l_{1}\cdot\cdot\cdot
l_{b-1}m l_{b+1}\cdot\cdot\cdot l_{r}}.
\end{equation*}

Let $M$ be an immersed hypersurface in $\mathbb{R}^{n+1}$. Denote
$R_{ijkl}$ to be the Riemannian curvature of $M\subset
\mathbb{R}^{n+1}$ with the induced metric $g$. Let $\nu$ be a given
unit normal and $h_{ij}$ be the second fundamental form of the
hypersurface with respect to $\nu$, that is
$$h_{ij}=-\langle\frac{\partial^2 X}{\partial x^i\partial x^j}, \nu\rangle_{\mathbb{R}^{n+1}}.$$
Recall the following identities
\begin{equation*}\label{Gauss for}
\nabla_i \nabla_j X=-h_{ij}\nu, \quad \quad \mbox{Gauss formula}
\end{equation*}
\begin{equation*}\label{Wein for}
\nabla_i \nu=h_{ij}X^j, \quad \quad \mbox{Weingarten formula}
\end{equation*}
\begin{equation*}\label{Gauss}
R_{ijkl}=h_{ik}h_{jl}-h_{il}h_{jk}, \quad \quad \mbox{Gauss
equation}
\end{equation*}
\begin{equation*}\label{Codazzi}
\nabla_{k}h_{ij}=\nabla_{j}h_{ik}, \quad \quad \mbox{Codazzi
equation}
\end{equation*}
where $X^j=g^{ij}\nabla_i X$.
Moreover, we have
\begin{equation}\label{2rd}
\begin{aligned}
\nabla_{i}\nabla_{j}h_{kl}
&=&\nabla_{k}\nabla_{l}h_{ij}+h^{m}_{j}(h_{il}h_{km}-h_{im}h_{kl})+h^{m}_{l}(h_{ij}h_{km}-h_{im}h_{kj}).
\end{aligned}
\end{equation}

\subsection{Star-shaped hypersurfaces in $\mathbb{R}^{n+1}$}

Let $M$ be a star-shaped hypersurface in $\mathbb{R}^{n+1}$ which can be represented by
\begin{equation*}
X(x)=\rho(x)x, \quad \mbox{for} \quad x \in \mathbb{S}^n,
\end{equation*}
where $X$ is the position vector of the hypersurface $M$ in $\mathbb{R}^{n+1}$ and $\rho$ is a smooth function on $\mathbb{S}^n$.
Let $\{e_1,...,e_n\}$ be a smooth local orthonormal frame field on
$\mathbb{S}^n$ and $e_\rho$ be the radial vector field
in $\mathbb{R}^{n+1}$. Let $D_i\rho=D_{e_i} \rho$, $D_iD_j\rho=D^2
\rho(e_i, e_j)$ denote the covariant derivatives of $\rho$ with respect
to the round metric $\sigma$ of $\mathbb{S}^n$. Then, the following formulas
hold:

(\romannumeral1) The tangential vector on $M$ is
\begin{equation*}
X_{i}=\rho e_{i}+D_i\rho e_{\rho}
\end{equation*}
and the corresponding outward unit normal vector is given by
\begin{equation}\label{Nor}
\nu=\frac{1}{v}\left(e_\rho-\frac{1}{\rho^2} D^j\rho e_j\right),
\end{equation}
where $v=\sqrt{1+\rho^{-2}|D \rho|^2}$ with $D^j \rho=\sigma^{ij}D_i\rho$.

(\romannumeral2) The induced metric $g$ on $M$ has the form
\begin{equation}\label{gij}
g_{ij}=\rho^2\sigma_{ij}+D_i\rho D_j\rho
\end{equation}
and its inverse is given by
\begin{equation}\label{gij2}
g^{ij}=\frac{1}{\rho^2}\left(\sigma^{ij}-\frac{D^i\rho D^j\rho}{\rho^2
v^{2}}\right).
\end{equation}

(\romannumeral3) The second fundamental form of $M$ is given by
\begin{equation}\label{hij}
h_{ij}=\frac{1}{v}\left(-D_iD_j\rho+\rho
\sigma_{ij}+\frac{2}{\rho}D_i\rho D_j\rho\right)
\end{equation}
and
\begin{equation*}\label{h_ij}
h^{i}_{j}=\frac{1}{\rho
v}\left(\delta^{i}_{j}+[-\sigma^{ik}+\frac{D^i\rho
D^k\rho}{\rho^2v^2}]D_jD_k(\log \rho)\right).
\end{equation*}

(\romannumeral4) For any $i, j, k=1, \cdots, n$, the following identities hold:
\begin{align}
(|X|^{2})_{i} &= 2\langle X, e_{i}\rangle, \label{xi} \\
(|X|^{2})_{ij} &= 2\delta_{ij} - 2h_{ij}\langle X, \nu\rangle, \label{xij} \\
\langle X, \nu\rangle_{i} &= h_{ik}\langle X, e_{k}\rangle, \notag \\
\langle X, \nu\rangle_{ij} &= h_{ijk}\langle X, e_{k}\rangle + h_{ij} - h_{ik}h_{kj}\langle X, \nu\rangle. \notag
\end{align}

 The principal curvature $(\kappa_1,\kappa_2,\ldots,\kappa_n)$ of $M$ are the eigenvalue of the second fundamental
form with respect to the metric satisfying the following equation:
$$\det(h_{ij}-\kappa g_{ij})=0.$$
Equation \eqref{G-eq} can be expressed as differential equations on the radial function
$\rho$ and position vector $X$, respectively. From \eqref{Nor} we have
$$\langle X,\nu\rangle=\rho^2(\rho^2+|\nabla \rho|^2)^{-1/2},$$
and equation \eqref{G-eq} is equivalent to
\begin{equation*}
\frac{\sigma_k}{\sigma_l}(\Lambda(\rho))=f \rho^{b+2q} (\rho^2 + |\nabla \rho|^2)^{-1/2}, \quad \text{on } \mathbb{S}^n.
\end{equation*}

\subsection{$k$-th elementary symmetric functions}
Let $\lambda=(\lambda_1,\cdots,\lambda_n)\in\mathbb{R}^n$, we recall
the definition of elementary symmetric functions for $1\leq k\leq n$,
\begin{equation*}\label{sigma}
\sigma_k(\lambda)= \sum _{1 \le i_1 < i_2 <\cdots<i_k\leq
n}\lambda_{i_1}\lambda_{i_2}\cdots\lambda_{i_k}.
\end{equation*}

\begin{definition}\label{defsigmak}
Let $1\leq k\leq n$ and $\Gamma_k$ be a cone in $\mathbb{R}^n$ determined by
$$\Gamma_k  = \{ \lambda  \in \mathbb{R}^n :\sigma _i (\lambda ) >
0,~\forall~ 1 \le i \le k\}.$$
\end{definition}
Denote $\sigma_{k-1}(\lambda|i)=\frac{\partial
\sigma_k}{\partial \lambda_i}$ and
$\sigma_{k-2}(\lambda|ij)=\frac{\partial^2 \sigma_k}{\partial
\lambda_i\partial \lambda_j}$, then we list some properties of
$\sigma_k$ which will be used later.

\begin{proposition}\label{sigma}
Let $\lambda=(\lambda_1,\cdots,\lambda_n)\in\mathbb{R}^n$ and $1\leq
k\leq n$. Then we have
\begin{enumerate}
\item[(\romannumeral1)]  $\Gamma_1\supset \Gamma_2\supset \cdot\cdot\cdot\supset
\Gamma_n$;

\item[(\romannumeral2)]  $\sigma_{k-1}(\lambda|i)>0$ for $\lambda \in \Gamma_k$ and
$1\leq i\leq n$;

\item[(\romannumeral3)]  $\sigma_k(\lambda)=\sigma_k(\lambda|i)
+\lambda_i\sigma_{k-1}(\lambda|i)$ for $1\leq i\leq n$;

\item[(\romannumeral4)]
$\sum_{i=1}^{n}\frac{\partial[\frac{\sigma_{k}}{\sigma_{l}}]^{\frac{1}{k-l}}}
{\partial \lambda_i}\geq [\frac{C^k_n}{C^l_n}]^{\frac{1}{k-l}}$ for
$\lambda \in \Gamma_{k}$ and $0\leq l<k$;

\item[(\romannumeral5)]  $\Big[\frac{\sigma_k}{\sigma_l}\Big]^{\frac{1}{k-l}}$ are
concave in $\Gamma_k$ for $0\leq l<k$;

\item[(\romannumeral6)] If $\lambda_1\geq \lambda_2\geq \cdot\cdot\cdot\geq \lambda_n$,
then $\sigma_{k-1}(\lambda|1)\leq \sigma_{k-1}(\lambda|2)\leq
\cdot\cdot\cdot\leq \sigma_{k-1}(\lambda|n)$ for $\lambda \in
\Gamma_k$;

\item[(\romannumeral7)]
$\sum_{i=1}^{n}\sigma_{k-1}(\lambda|i)=(n-k+1)\sigma_{k-1}(\lambda)$;

\item[(\romannumeral8)]
If $\lambda_1\geq \lambda_2\geq \cdot\cdot\cdot\geq \lambda_n$,
then
$\lambda_1 \sigma_{k-1}(\lambda |1) \geq \frac{k}{n}\sigma_k(\lambda).$
\end{enumerate}
\end{proposition}
\begin{proof}
All the properties are well known. For example, see Chapter XV in
\cite{Li96} or \cite{Hui99} for proofs of (\romannumeral1), (\romannumeral2), (\romannumeral3), (\romannumeral6), (\romannumeral7) and
(\romannumeral8); see Lemma 2.2.19 in \cite{Ger06} for the proof of (\romannumeral4); see \cite{CNS85} and \cite{Li96} for the proof of (\romannumeral5).
\end{proof}

\begin{proposition}
Let $A=A_{ij}$ be an $n\times n$ symmetric matrix, $\lambda(A)=(\lambda_1,\lambda_2,\cdots,\lambda_n)$ be the eigenvalues of the symmetric matric $A$. Suppose that $A=A_{ij}$ is diagonal, then
\begin{equation*}
\begin{aligned}
&\frac{\partial \lambda_i}{\partial A_{ii}}= 1, \quad \frac{\partial \lambda_k}{\partial A_{ij}} = 0 \text{ otherwise}, \\
&\frac{\partial^2 \lambda_i}{\partial A_{ij} \partial A_{ji}}= \frac{1}{\lambda_i - \lambda_j}, \quad i \ne j \text{ and } \lambda_i \ne \lambda_j,  \\
&\frac{\partial^2 \lambda_i}{\partial A_{kl} \partial A_{pq}}= 0 \text{ otherwise}.
\end{aligned}
\end{equation*}
\end{proposition}

\begin{proposition}\label{diagonal}
Suppose $A=A_{ij}$ is diagonal and \(m\) is a positive integer with \(1\le m\le n\), then
\begin{equation*}
\frac{\partial \sigma_m(A)}{\partial A_{ij}}=
  \begin{cases}
  \sigma_{m-1}(A|i), ~&i=j,\\
  0, ~~~&otherwise,
  \end{cases}
\end{equation*}
\begin{equation*}
\frac{\partial^2\sigma_m(A)}{\partial A_{ij}\partial A_{pq}}=
\begin{cases}
\sigma_{m-2}(A|ip), ~~~&i=j,p=q,i\neq p,\\
-\sigma_{m-2}(A|ip), ~~~&i=q,j=p,i\neq j,\\
0,~~~&\mbox{otherwise}.
\end{cases}
\end{equation*}
where  $\sigma_k(A |i)$  is the symmetric function with $A$ deleting the $i$-row and $i$-column and $\sigma_k(A|ij)$ is  the symmetric function with $A$ deleting the $i,j$-rows and $i,j$-columns.
\end{proposition}

The generalized Newton-MacLaurin inequality is as follows.
\begin{proposition}\label{NM}
For $\lambda \in \Gamma_m$ and $m > l \geq 0$, $ r > s \geq 0$, $m
\geq r$, $l \geq s$, then
\begin{align}
\Bigg[\frac{{\sigma _m (\lambda )}/{C_n^m }}{{\sigma _l (\lambda
)}/{C_n^l }}\Bigg]^{\frac{1}{m-l}} \le \Bigg[\frac{{\sigma _r
(\lambda )}/{C_n^r }}{{\sigma _s (\lambda )}/{C_n^s
}}\Bigg]^{\frac{1}{r-s}}. \notag
\end{align}
\end{proposition}
\begin{proof}
See \cite{S05}.
\end{proof}

\subsection{Derivations on the exterior algebra and properties for $\frac{\sigma_{k}}{\sigma_{l}}(\Lambda)$}\
Let $\mathbf{p} \in \{1, \cdots, n\}$ and $n\geqslant 2$. We use the standard notation for ordered multi-indices
$$ \mathfrak{J}(\mathbf{p},n):=\{I=(i_{1},\cdots,i_{\mathbf{p}})\mid i_{s}~integers~and~1\leqslant i_{1}<\cdots<i_{\mathbf{p}}\leqslant n\}. $$
Set $\mathfrak{J}(0,n)=\{0\}$ and $|I|=\mathbf{p}$ if $I\in \mathfrak{J}(\mathbf{p},n)$. For $I\in \mathfrak{J}(\mathbf{p},n)$,
\begin{enumerate}
\item[(\romannumeral1)]  $\overline{I}$ is the element in $\mathfrak{J}(n - \mathbf{p},n)$ which complements $I$ in $\{1,2,\cdots,n\}$ in the natural increasing order.
\item [(\romannumeral2)] $I-i$ means the multi-index of length $\mathbf{p}-1$ obtained by removing $i$ from $I$ for any $i \in I$.
\item [(\romannumeral3)] $I+j$ means the multi-index of length $\mathbf{p}+1$ obtained by adding $j$ to $I$  for any $j \notin I$.
\item [(\romannumeral4)] $\sigma(I,J)$ is the sign of the permutation which reorders $(I,J)$ in the natural increasing order for any
multi-index $J$ with $I\cap J=\emptyset$. In particular set $\sigma(\overline{0},0):=1$.
\end{enumerate}

\begin{definition}
Let $\mathbf{p} \in \{1, \cdots, n\}$, the $(\mathbf{p}, k)$-cone is defined by
$$\mathcal{P}_{\mathbf{p}, k}=\{\lambda=(\lambda_1, \lambda_2, \cdots, \lambda_n)\in \mathbb{R}^n\mid \sigma_j(\Lambda) >0, \forall~~1\leq j\leq k\},$$
where $\Lambda=(\Lambda_{I_1}, \Lambda_{I_2}, \cdots, \Lambda_{I_N}) \in \mathbb{R}^N$,  and $\Lambda_I=\lambda_{i_1}+\lambda_{i_2}+\cdots+\lambda_{i_\mathbf{p}}$
for any $I=(i_1, i_2, \cdots, i_p)\in \mathfrak{J}(\mathbf{p},n)$.
\end{definition}

Hence we have the following proposition.
\begin{proposition}
A $C^2$-hypersurface $M\subset \mathbb{R}^{n+1}$ is $(\mathbf{p},k)$-convex if and only if  the principal curvatues $\kappa(X) \in \mathcal{P}_{\mathbf{p}, k}$ for any $X\in M$.
\end{proposition}

\begin{definition}
Let $\mathbf{p} \in \{1, \cdots, n\}$ and  $A=\{a_{ij}\}$ be a symmetric matrix. The linear map \(D_A:\Lambda^p\mathbb R^n\to\Lambda^p\mathbb R^n\) induced by \(A\) is defined by
\begin{equation*}
\begin{aligned}
\mathcal{D}_A:&~~~~~~~\Lambda^\mathbf{p}\mathbb{R}^n &&\longrightarrow \Lambda^\mathbf{p}\mathbb{R}^n \\
&v_1\wedge\cdots\wedge v_\mathbf{p} &&\longmapsto \sum_{i=1}^{\mathbf{p}} v_1 \wedge \cdots \wedge v_{i-1}\wedge (Av_i)\wedge v_{i+1}  \cdots\wedge v_\mathbf{p}.
\end{aligned}
\end{equation*}
\end{definition}

 Fix an orthonormal basis $(e_{1},\cdot\cdot\cdot,e_{n})$ of $\mathbb{R}^n$ and the corresponding basis $\{e_{I}\}_{I\in\mathfrak{J}(\mathbf{P},n)}$ of $\Lambda^\mathbf{P}\mathbb{R}^n$, where $e_I:=e_{i_1}\wedge\cdots\wedge e_{i_\mathbf{P}}$ for any $I=(i_1,\cdots i_\mathbf{P}) \in \mathfrak{J}(\mathbf{P},n)$. Obviously, the eigenvalues of    $\mathcal{D}_A$ can be written as
 $\Lambda_{I}=(\Lambda_{I_1}, \Lambda_{I_2}, \cdots, \Lambda_{I_N})$ if  $\lambda=(\lambda_1,\lambda_2,\cdots,\lambda_n)$ are eigenvalues of symmetric matrix $A$.


It is worth emphasizing that given any orthonormal basis of $\mathbb{R}^n$, $\mathcal{D}_A$ has a matrix representation with respect to the induced basis which has components being linear combinations of the entries of $A$.

\begin{proposition}\label{wij}
Let $A=\{a_{ij}\}$ be a symmetric matrix, the corresponding matrix $W:=\{ W_{IJ}\}_{I,J\in \mathfrak{J}(\mathbf{p},n)}$ of linear derivation $\mathcal{D}_A$
in the canonical basis $\{ e_{I_1}, e_{I_2}, \cdots, e_{I_N}\}$ reads
\begin{equation}\label{propw}
W_{IJ}=
\begin{cases}
\sum_{i\in I}a_{ii}, ~~~&I=J,\\
\sigma(i,I-i)\sigma(j,J-j)a_{ij}, ~~~&I=i+K, J=j+K, |K|=\mathbf{p}-1, i\neq j,\\
0,~~~&\mbox{otherwise}.
\end{cases}
\end{equation}
and hence we have
\begin{equation*}
\frac{\partial W_{IJ}}{\partial a_{ij}}=
\begin{cases}
1, ~~~&i=j,I=J,i\in I,\\
\sigma(i,I-i)\sigma(j,J-j), ~~~&I=i+K,J=j+K,|K|=\mathbf{p}-1,i\neq j,\\
0,~~~&\mbox{otherwise}.
\end{cases}
\end{equation*}
\end{proposition}

\begin{proof}
See Proposition 2.7 in \cite{GongTu2026}.
\end{proof}

Then we have the following proposition.
\begin{proposition}\label{f24}
Let $\kappa = (\kappa_1, \ldots, \kappa_n)$ be the principal curvature of $M$. Suppose that $M\subset \mathbb{R}^{n+1}$ is $(\mathbf{p},k)$-convex, $\Lambda(\kappa) = (\Lambda_{I_1}, \ldots, \Lambda_{I_N})$ and $0 \leq l < k \leq N$, then the operator
$\left(\frac{\sigma_k}{\sigma_l}(\Lambda(\kappa)) \right)^{\frac{1}{k-l}}$
is elliptic and concave with respect to $\kappa$. Moreover we have
$$\sum_{i=1}^n \frac{\partial \left(\frac{\sigma_k}{\sigma_l}(\Lambda(\kappa)) \right)^{\frac{1}{k-l}}}{\partial \kappa_i} \geq \mathbf{p}\left( \frac{C_N^k}{C_N^l}\right)^{\frac{1}{k-l}}>0.$$
\end{proposition}
\begin{proof}
See Lemma 2.5 in \cite{Zhou-2024}.
\end{proof}

Let $\kappa=(\kappa_1,\cdots,\kappa_n)$ are the eigenvalues of the matrix $\{a_{ij}\}$, $W=\{W_{IJ}\}$ be the transition matrix of liear derivation  $\mathcal{D}_A$ and $\Lambda= (\Lambda_1,\cdots,\Lambda_N)$ be the the eigenvalues of $W$.
For convenience, we make the following notations for $\frac{\sigma_{k}}{\sigma_{l}}(\Lambda)$
\begin{equation*}
\begin{aligned}
&F(W)=\frac{\sigma_k}{\sigma_l}(\Lambda(W)), \quad F^{ij}=\frac{\partial {F}}{\partial a_{ij}}, \quad {F}^{ij, rs}=\frac{\partial^2 {F}}{\partial a_{ij} \partial a_{rs}}.\\
\end{aligned}
\end{equation*}
The following lemma can be found in \cite{ge}.

\begin{proposition}\label{f25}
For any symmetric matrix $(\eta_{ij})$, we have
$$
F^{ij,kl} \eta_{ij} \eta_{kl} = F^{ii,jj} \eta_{ii} \eta_{jj} + \sum_{i \neq j} \frac{F^{ii} - F^{jj}}{a_{ii} - a_{jj}} \eta_{ij}^2.
$$
\end{proposition}

Finally, we prove the ``inverse convexity'' of the operator $-\left[\frac{\sigma_{k}}{\sigma_{l}}(\Lambda(W))\right]^{-\frac{1}{k-l}}$, which is crucial to establish the Constant Rank Theorem.

\begin{proposition}\label{L1}
Let $\kappa=(\kappa_1,\cdots,\kappa_n)\in {\Gamma}_{n}$ be the eigenvalues of the matrix $A=\{a_{ij}\}$.
$W=\{W_{IJ}\}$ is the symmetric matrix of the linear derivation $\mathcal{D}_A$ associated with $A$, as defined in \eqref{propw}.
Then $-\left[\frac{\sigma_{k}}{\sigma_{l}}(\Lambda(W))\right]^{-\frac{1}{k-l}}$ is ``inverse convex" respect to the matrix
$\{a_{ij}\}$, that is, for any $n \times n$ symmetric matrix $\{\xi_{ij}\}$,
\begin{equation}\label{in-1}
\sum_{i,j,r,s}\left(\frac{\partial^{2}\left(-\left[\frac{\sigma_{k}}{\sigma_{l}}(\Lambda(W))\right]^{-\frac{1}{k-l}}\right)}{\partial a_{ij}\partial a_{rs}}
+2\frac{\partial\left(-\left[\frac{\sigma_{k}}{\sigma_{l}}(\Lambda(W))\right]^{-\frac{1}{k-l}}\right)}{\partial a_{ir}}a^{js}\right)\xi_{ij}\xi_{rs} \geq 0,
\end{equation}
where $\{a^{ij}\}=\{a_{ij}\}^{-1}. $
\end{proposition}
\begin{proof}
Since $\kappa \in {\Gamma}_n$, then $\Lambda \in  {\Gamma}_N$, and $W$ is positive definite.
It is well known that the function $-\left[\frac{\sigma_{k}}{\sigma_{l}}(\Lambda(W))\right]^{-\frac{1}{k-l}}$ is ``inverse convex" with respect to $W$; see Lemma 3.4 of \cite{BA-07}. Consequently, for any $N\times N$ symmetric   matrix $\{\xi_{IJ}\}$, the following inequality holds:
\begin{equation}\label{in-2}
\sum_{I,J,R,S\in \mathfrak{J}}\left(\frac{\partial^{2}\left(-\left[\frac{\sigma_{k}}{\sigma_{l}}(\Lambda(W))\right]^{-\frac{1}{k-l}}\right)}{\partial W_{IJ}\partial W_{RS}}
+2\frac{\partial\left(-\left[\frac{\sigma_{k}}{\sigma_{l}}(\Lambda(W))\right]^{-\frac{1}{k-l}}\right)}{\partial W_{IR}}W^{JS}\right)\xi_{IJ}\xi_{RS}\geq0.
\end{equation}
where $\{W^{JS}\}=\{W_{JS}\}^{-1}. $
For convenience, introduce the notations
\begin{equation*}
\begin{aligned}
\widetilde{G}^{IJ}=\frac{\partial {\left(-\left[\frac{\sigma_{k}}{\sigma_{l}}(\Lambda(W))\right]^{-\frac{1}{k-l}} \right)}}{\partial W_{IJ}}, \quad \widetilde{G}^{IJ, RS}=\frac{\partial^2 {{\left(-\left[\frac{\sigma_{k}}{\sigma_{l}}(\Lambda(W))\right]^{-\frac{1}{k-l}} \right)}}}{\partial W_{IJ} \partial W_{RS}}.
\end{aligned}
\end{equation*}

Suppose $\{a_{ij}\}$ is diagonal, then so is $W$. It follows from \eqref{in-2}, Proposition~\ref{diagonal} and Proposition~\ref{wij}  that
\begin{align*}
&\sum_{i,j,r,s}\left(\frac{\partial^{2}\left(-\left[\frac{\sigma_{k}}{\sigma_{l}}(\Lambda(W))\right]^{-\frac{1}{k-l}}\right)}{\partial a_{ij}\partial a_{rs}}
+2\frac{\partial\left(-\left[\frac{\sigma_{k}}{\sigma_{l}}(\Lambda(W))\right]^{-\frac{1}{k-l}}\right)}{\partial a_{ir}}a^{js}\right)\xi_{ij}\xi_{rs} \\
=&\sum_{I,J,R,S\in \mathfrak{J}}\widetilde{G}^{IJ, RS}\left(\sum_{i,j}\frac{\partial W_{IJ}}{\partial a_{ij}}\xi_{ij}\right)\left(\sum_{r,s}\frac{\partial W_{RS}}{\partial a_{rs}}\xi_{rs}\right)
+2\sum_{i,j,r,s}\sum_{P,Q\in \mathfrak{J}}\widetilde{G}^{PQ}\frac{\partial W_{PQ}}{\partial a_{ir}}\xi_{ij}\xi_{rs}{a^{js}}\\
\geq&-2\sum_{I,J,R,S\in \mathfrak{J}}\widetilde{G}^{IR}W^{JS}\left(\sum_{i,j}\frac{\partial W_{IJ}}{\partial a_{ij}}\xi_{ij}\right)\left(\sum_{r,s}\frac{\partial W_{RS}}{\partial a_{rs}}\xi_{rs}\right)
+2\sum_{i,j,r,s}\sum_{P,Q\in \mathfrak{J}}\widetilde{G}^{PQ}\frac{\partial W_{PQ}}{\partial a_{ir}}\xi_{ij}\xi_{rs}{a^{js}}\\
=&-2\sum_{I,J\in \mathfrak{J}}\widetilde{G}^{II}\frac{\left(\sum_{i,j}\frac{\partial W_{IJ}}{\partial a_{ij}}\xi_{ij}\right)^2}{W_{JJ}}
+2\sum_{i,j,r}\sum_{P\in \mathfrak{J}}\widetilde{G}^{PP}\frac{\partial W_{PP}}{\partial a_{ir}}\frac{\xi_{ij}\xi_{rj}}{a_{jj}}\\
=&-2\sum_{I\in \mathfrak{J}} \widetilde{G}^{II}\frac{\left({\sum_i\frac {\partial W_{II}}{\partial a_{ii}}\xi_{ii}}\right)^2}{{W_{II}}}
-2\sum_{I\in \mathfrak{J}} \widetilde{G}^{II}\sum _{|J\cap I|=\mathbf{p}-1}\frac{\left({\sum_{i,j}\frac {\partial W_{IJ}}{\partial a_{ij}}\xi_{ij}}\right)^2}{W_{JJ}}
+2\sum_{I\in \mathfrak{J}} \widetilde{G}^{II}\sum_{i,j}\frac {\partial W_{II}}{\partial a_{ii}}\frac {\xi_{ij}^2}{a_{jj}}\\
=&2\sum_{I\in \mathfrak{J}} \widetilde{G}^{II}\left\{-\frac{\left(\sum_{i\in I}\xi_{ii}\right)^2}{\sum_{i\in I}a_{ii}}-\sum_{i\in I}\sum_{j\in{\bar I}}\frac{\xi_{ij}^2}{\sum_{s\in I-i+j}a_{ss}}
+\sum_{i\in I}\frac{\xi_{ii}^2}{a_{ii}}+\sum_{i\in I}\left(\sum_{j\in {I-i}}\frac{\xi_{ij}^2}{a_{jj}}+\sum_{j\in {\bar I}}\frac{\xi_{ij}^2}{a_{jj}}
\right)\right\}\\
\geq&2\sum_{I\in \mathfrak{J}} \widetilde{G}^{II}\left\{-\frac{\left(\sum_{i\in I}\frac{\xi_{ii}}{\sqrt{a_{ii}}}\sqrt{a_{ii}}\right)^2}{\sum_{i\in I}a_{ii}}+\sum_{i\in I}\frac{\xi_{ii}^2}{a_{ii}}\right\}\\
\geq&2\sum_{I\in \mathfrak{J}} \widetilde{G}^{II}\left\{-\frac{\sum_{i\in I}\frac{\xi_{ii}^2}{{a_{ii}}}\sum_{i\in I}{a_{ii}}}{\sum_{i\in I}a_{ii}}+\sum_{i\in I}\frac{\xi_{ii}^2}{a_{ii}}\right\}\\
\geq&0.
\end{align*}
Then \eqref{in-1} holds.
\end{proof}


\section{Existence and uniqueness in nonhomogeneous case}
In this section, we obtain the existence and uniqueness of $(\mathbf{p}, k)$-convex (or strictly convex) hypersurface for  curvature equation \eqref{G-eq} in case $-b-q-k+l>0$. In order to prove the main result, a priori estimates and the Constant Rank Theorem are establish as follows.

\subsection{$C^1$ estimates}

\begin{theorem}\label{c0}
Let $M\subset \mathbb{R}^{n+1}$ be a closed, star-shaped and   $(\mathbf{p}, k)$-convex hypersurface satisfying  equation \eqref{G-eq} with a positive function $f \in C^2(\mathbb{S}^n)$.
If $-b-q-k+l>0$, then
$$
\left(\frac{C_N^l \min\limits_{ \mathbb{S}^n} f}{C_N^k \mathbf{p}^{k-l}}\right)^{\frac{1}{-b-q-k+l}}
\leq \rho(x)
\leq \left(\frac{C_N^l \max\limits_{ \mathbb{S}^n} f}{C_N^k \mathbf{p}^{k-l}}\right)^{\frac{1}{-b-q-k+l}},       \quad \forall~x \in \mathbb{S}^n.
$$
\end{theorem}

\begin{proof}
Assume $\rho$ attains its maximum value at $x_0 \in \mathbb{S}^n$, then
$$\nabla \rho(x_0) = 0  \quad \text{and} \quad \nabla^2 \rho(x_0) \leq 0.$$
From \eqref{gij}-\eqref{hij}, the principal curvatures of $M$ at $x_0$ satisfy
$$\kappa_i(x_0) \geq \frac{1}{\max\limits_{\mathbb{S}^n} \rho}, \quad \text{for } 1 \leq i \leq n.$$
Hence
$$\mathbf{p}^{k-l} \frac{C_N^k}{C_N^l} \left( \frac{1}{\max\limits_{\mathbb{S}^n} \rho} \right)^{k-l}
\leq\frac{\sigma_k(\Lambda(\kappa))}{\sigma_l(\Lambda(\kappa))}
=f \rho^{b+2q} (\rho^2 + |\nabla \rho|^2)^{-\frac{q}{2}}
\leq (\max_{ \mathbb{S}^n} f)(\max_{\mathbb{S}^n} \rho)^{b+q},$$
which implies
$$(\max_{\mathbb{S}^n} \rho)^{-b-q-k+l} \leq \frac{C_N^l}{\mathbf{p}^{k-l}C_N^k } \cdot \max_{ \mathbb{S}^n} f.$$
Similar argument yields
$$\frac{C_N^l \min\limits_{ \mathbb{S}^n} f}{\mathbf{p}^{k-l}C_N^k } \leq (\min_{\mathbb{S}^n} \rho)^{-b-q-k+l}.$$
\end{proof}

In order to get the gradient estimate, we  set  $u=-\log \rho$. Then the first and second fundamental form of $M$  is
\begin{equation*}
\begin{aligned}
g_{i j} & =e^{-2 u}\left(\delta_{i j}+u_{i} u_{j}\right), \\
h_{i j} & =e^{-u}\left(1+|\nabla u|^{2}\right)^{-\frac{1}{2}}\left(\delta_{i j}+u_{i} u_{j}+u_{i j}\right),
\end{aligned}
\end{equation*}
and
$$\left[g^{i j}\right]^{\frac{1}{2}}=e^{u}\Big(\delta_{i j}-\frac{u_{i} u_{j}}{\sqrt{1+|\nabla u|^{2}}\big(1+\sqrt{1+|\nabla u|^{2}}\big)}\Big).$$
If we set
\begin{equation*}
\begin{aligned}
\bar{g}^{i j} & =\Big(\delta_{i j}-\frac{u_{i} u_{j}}{\sqrt{1+|\nabla u|^{2}}\big(1+\sqrt{1+|\nabla u|^{2}}\big)}\Big), \\
\bar{h}_{l m} & =\delta_{l m}+u_{l} u_{m}+u_{l m}, \\
a_{i j} & =\bar{g}^{i l} \bar{h}_{l m} \bar{g}^{m j},
\end{aligned}
\end{equation*}
then the symmetric matrix  $\{A_{i j}\}=\{[g^{i k}]^{\frac{1}{2}}{h}_{k l}[g^{l j}]^{\frac{1}{2}}\}$  behaves as
$$A_{i j}=e^{u}\left(1+|\nabla u|^{2}\right)^{-\frac{1}{2}} a_{i j},$$
and equation \eqref{G-eq} is equivalent to
\begin{equation}\label{c121}
\begin{aligned}
F=\frac{\sigma_{k}}{\sigma_{l}}(\Lambda(a_{ij}))=\frac{\sigma_{k}}{\sigma_{l}}(\Lambda(W))=fe^{(-b-q-k+l)u}\left(1+|\nabla u|^{2}\right)^{\frac{k-l-q}{2}},
\end{aligned}
\end{equation}
where $W$ is the symmetric matrix associated with  $\{a_{ij}\}$, as defined in \eqref{propw}.

\begin{theorem}\label{c1-877}
Let $0 \leq l <k\leq n$, $\mathbf{p}>1$,  $M\subset \mathbb{R}^{n+1}$ be a closed, star-shaped and  $(\mathbf{p}, k)$-convex hypersurface satisfying the equation \eqref{G-eq} with a positive function $f\in C^{2}(\mathbb{S}^{n})$.  If $-b-q-k+l>0$, then
$$\max_{\mathbb{S}^{n}}|\nabla \log\rho|\leq C,$$
where the constant $C$ depends on $b, q, \mathbf{p}, k,l$ and $\max_{\mathbb{S}^{n}}\frac{|\nabla f|}{f}$. In particular, we have
$$1\leq\frac{\max_{\mathbb{S}^{n}}\rho}{\min_{\mathbb{S}^{n}}\rho}\leq e^C.$$
\end{theorem}

\begin{proof}
Define the function $Q=|\nabla u|^{2}$, and assume that $Q$ attains its maximum at a point $x_{0}$. Choosing a local coordinate frame field around $x_{0}$, such that $$u_{1}=|\nabla u|,\quad \text{and}~~\{u_{ij}\}_{2\leq i,j\leq n}~\text{is ~ diagonal}.$$
Next, all the calculations are performed at the point $x_0$. First
$$0=Q_{i}=2\sum_mu_{m}u_{mi}=2u_{1}u_{1i},$$
which implies
$$ u_{1i}=0~~~\text{for all}~~~i=1,\cdots,n.$$

Then $\{u_{ij}\}$ is diagonal at $x_{0}$ and the matrix $\{\bar{g}^{ij}\},\{\bar{h}_{lm}\},$ and $\{a_{ij}\}$ are also diagonal at $x_{0}$,
$$ \bar{g}^{11}=\frac{1}{\sqrt{1+|\nabla u|^{2}}},\quad\bar{h}_{11}=1+|\nabla u|^{2},\quad a_{11}=1, $$
and for $2\leq i\leq n,$
$$\bar{g}^{ii}=1,\quad\bar{h}_{ii}=a_{ii}=1+u_{ii}.$$
Then ${F}^{ij}:=\frac{\partial\left[\frac{\sigma_{k}}{\sigma_{l}}(\Lambda(W))\right]}{\partial a_{ij}}$ is diagonal at $x_{0}$. Differentiating equation \eqref{c121} on both sides,
\begin{equation}\label{c111a}
\begin{aligned}
{F}^{ij}a_{ijs}=e^{(-b-q-k+l)u}(1+|\nabla u|^{2})^{\frac{k-l-q}{2}}\left[(-b-q-k+l)fu_{s}+f_{s}\right],
\end{aligned}
\end{equation}
On the other hand,  $(\bar{g}^{il})_s=0$, therefore,
$$u_{s}a_{ijs}=u_{s}(\bar{g}^{il}\bar{h}_{lm}\bar{g}^{mj})_{s}=\bar{g}^{il}u_{s}\bar{h}_{lms}\bar{g}^{mj}=\bar{g}^{il}u_{s}u_{lms}\bar{g}^{mj}.$$
From the above, we derive
\begin{equation}\label{c111b}
\begin{aligned}
\sum^n_{s=1}u_{s}{F}^{ij}a_{ijs}=&\sum^n_{s=1}{F}^{ij}\bar{g}^{il}u_{s}u_{lms}\bar{g}^{mj}\\
=&\sum^n_{s=1}{F}^{ij}\bar{g}^{il}u_{s}\left(u_{slm}-u_{s}\delta_{lm}+u_{l}\delta_{sm}\right)\bar{g}^{mj}\\
=&\sum^n_{s=1}{F}^{ij}\bar{g}^{il}\bar{g}^{mj}u_{s}u_{slm}-|\nabla u|^{2}{F}^{ij}\bar{g}^{il}\bar{g}^{lj}+{F}^{ij}\bar{g}^{il}\bar{g}^{mj}u_{m}u_{l}.
\end{aligned}
\end{equation}
Set $\widetilde{G}^{ij}={F}^{lm}\overline{g}^{il}\overline{g}^{mj}$, combining \eqref{c111a} and \eqref{c111b}, we obtain
$$\sum^n_{s=1}\widetilde{G}^{lm}u_{s}u_{slm}=|\nabla u|^{2}\sum_{i=2}^{n}\widetilde{G}^{ii}+e^{(-b-q-k+l)u}(1+|\nabla u|^{2})^{\frac{k-l-q}{2}}\left[(-b-q-k+l)f|\nabla u|^{2}+|\nabla u|f_{1}\right].$$
Further computation gives
\begin{align}
 0\geq \frac{1}{2}\widetilde{G}^{ij}Q_{ij}=&\sum^n_{m=1}\widetilde{G}^{ij}u_{mi}u_{mj}+\sum^n_{m=1}\widetilde{G}^{ij}u_{m}u_{mij} \notag\\
 \geq&\sum_{i=2}^{n}\widetilde{G}^{ii}u_{ii}^{2}+|\nabla u|^{2}\sum_{i=2}^{n}\widetilde{G}^{ii}+e^{(-b-q-k+l)u}(1+|\nabla u|^{2})^{\frac{k-l-q}{2}}f_{1}|\nabla u|\notag\\
 &+(-b-q-k+l)e^{(-b-q-k+l)u}(1+|\nabla u|^{2})^{\frac{k-l-q}{2}}f|\nabla u|^{2}\notag\\
 \geq &\sum_{i=2}^{n}{F}^{ii}u_{ii}^{2}+|\nabla u|^{2}\sum_{i=2}^{n}{F}^{ii}- F\frac{|\nabla f|}{f}|\nabla u|+(-b-q-k+l)F|\nabla u|^{2}\notag\\
 =&\sum_{i=2}^{n}{F}^{ii}u_{ii}^{2}+(-b-q-k+l)F\left(|\nabla u|-\frac{|\nabla f|}{2(-b-q-k+l)f}\right)^2\notag\\
 &-\frac{|\nabla f|^2F}{4(-b-q-k+l)f^2}+|\nabla u|^{2}\sum_{i=2}^{n}{F}^{ii}\notag\\
\geq& ~|\nabla u|^{2}\sum_{i=2}^{n}{F}^{ii}+\sum_{i=2}^{n}{F}^{ii}u_{ii}^{2}-\frac{|\nabla f|^2F}{4(-b-q-k+l)f^2}\label{equiiur-01},
\end{align}
since $\widetilde{G}^{ii}={F}^{ii}$ for $2\leq i\leq n$ and $-b-q-k+l>0$.
At $x_0$, the eigenvalues of matrix $\{a_{ij}\}$ satisfy
$$\kappa_{1}=1,\quad\kappa_{2}=1+u_{22},\quad\cdots,\quad\kappa_{n}=1+u_{nn}.$$
Without loss of generality, assume  that $\kappa_2\geq \kappa_3\geq \cdots\geq \kappa_n$.
Set $G^{II}:=\frac{\partial \left[\frac{\sigma_k}{\sigma_l}(\Lambda)\right]}{\partial \Lambda_I}$, then
\begin{align}
&~|\nabla u|^{2}\sum_{i=2}^{n}{F}^{ii}+\sum_{i=2}^{n}{F}^{ii}u_{ii}^{2}\label{equiiur-001}\\
=&~|\nabla u|^{2} \left( \sum_{\{I|1\notin I\}}  \sum_{i\in I} G^{II} + \sum_{\{I|1\in I\}}  \sum_{i\in I, i\neq1} G^{II}\right) + \sum_{\{I|1\notin I\}} G^{II} \sum_{i\in I} \kappa_i^2-2\sum_{\{I|1\notin I\}} G^{II} \sum_{i\in I} \kappa_i \notag\\
&+\sum_{\{I|1\notin I\}}  \sum_{i\in I} G^{II}+ \sum_{\{I|1\in I\}} G^{II} \sum_{i\in I, i\neq1} \kappa_i^2-2\sum_{\{I|1\in I\}} G^{II} \sum_{i\in I, i\neq1} \kappa_i +\sum_{\{I|1\in I\}}  \sum_{i\in I, i\neq1} G^{II}\notag\\
\geq&~\mathbf{p} (1+|\nabla u|^2)  \sum_{\{I|1\notin I\}} G^{II}+(\mathbf{p}-1) (1+|\nabla u|^2) \sum_{\{I|1\in I\}} G^{II}+\frac{1}{\mathbf{p}} \sum_{\{I|1\notin I\}} G^{II} \Lambda_I^2\notag\\
&-2\sum_{\{I|1\notin I\}} G^{II}\Lambda_I+\frac{1}{\mathbf{p}} \sum_{\{I|1\in I\}} G^{II} \Lambda_I^2- \sum_{\{I|1\in I\}} G^{II}-2\sum_{\{I|1\in I\}} G^{II}\Lambda_I+2 \sum_{\{I|1\in I\}} G^{II}
\notag\\
=&~\frac{1}{\mathbf{p}} \sum_{I} G^{II} \Lambda_I^2 -2\sum_{I} G^{II}\Lambda_I+(\mathbf{p}-1) (1+|\nabla u|^2) \sum_{I} G^{II}+ \sum_I G^{II} +|\nabla u|^2 \sum_{\{I|1\notin I\}} G^{II}\notag.
\end{align}
Applying Proposition \ref{f24}, we derive
\begin{equation*}\label{equiiur-011111111}
\begin{aligned}
\sum_{I} G^{II}=\frac{1}{\mathbf{p}}\sum_{i} F^{ii}
=&\frac{k-l}{\mathbf{p}} F^{1-\frac{1}{k-l}}
\sum_{i} \frac{\partial \left(\frac{\sigma_k}{\sigma_l}(\Lambda(\kappa)) \right)^{\frac{1}{k-l}}}{\partial \kappa_i} \\
\geq&(k-l) \left( \frac{C_N^k}{C_N^l}\right)^{\frac{1}{k-l}}F^{1-\frac{1}{k-l}}.
\end{aligned}
\end{equation*}
On the other hand, a direct computation  shows that
\begin{equation}\label{equiiur-011}
\begin{aligned}
 \sum_{I} G^{II} \Lambda_I^2
=&(l+1)\frac{\sigma_k(\Lambda)\sigma_{l+1}(\Lambda)}
{\sigma_l^2(\Lambda)}-(k+1)\frac{\sigma_{k+1}(\Lambda)}{\sigma_l(\Lambda)}\\
\geq&(N-l)\left(\frac{C_N^l}{C_N^k}\right)^{\frac{1}{k-l}} F^{1+\frac{1}{k-l}}-(N-k)\left(\frac{C_N^l}{C_N^k}\right)^{\frac{1}{k-l}} F^{1+\frac{1}{k-l}}\\
=&(k-l)\left(\frac{C_N^l}{C_N^k}\right)^{\frac{1}{k-l}} F^{1+\frac{1}{k-l}}.
\end{aligned}
\end{equation}
Putting \eqref{equiiur-001} - \eqref{equiiur-011} into \eqref{equiiur-01}, we obtain
\begin{equation*}\label{c11111}
\begin{aligned}
 0\geq&\frac{k-l}{\mathbf{p}}\left(\frac{C_N^l}{C_N^k}\right)^{\frac{1}{k-l}} F^{\frac{1}{k-l}}
+ (\mathbf{p}-1)(k-l) \left(\frac{C_N^k}{C_N^l}\right)^{\frac{1}{k-l}}(1+|\nabla u|^2) F^{-\frac{1}{k-l}} \\
&-\frac{|\nabla f|^2}{4(-b-q-k+l)f^2}-2(k-l)\\
\geq&2(k-l)\sqrt{\frac{\mathbf{p}-1}{\mathbf{p}}}(1+|\nabla u|^{2})^{\frac{1}{2}}-\frac{|\nabla f|^2}{4(-b-q-k+l)f^2}-2(k-l),
\end{aligned}
\end{equation*}
the last inequality follows from the Cauchy-Schwarz inequality.
Hence
$$|\nabla u| \leq C,$$
where $C$ depends on $b, q, \mathbf{p}, k,l$ and $\max_{\mathbb{S}^{n}} \frac{|\nabla f|}{f}$. Let $\widetilde{\rho}=\frac{\rho}{\min_{\mathbb{S}^{n}} \rho}$, since $|\nabla \log \widetilde{\rho}|=|\nabla \log \rho| \leq C$ and $\min_{\mathbb{S}^{n}} \log \widetilde{\rho}=0$, we obtain
$$1\leq\max_{\mathbb{S}^{n}} \widetilde{\rho}=\frac{\max_{\mathbb{S}^{n}}\rho}{\min_{\mathbb{S}^{n}} \rho} \leq e^C.$$
\end{proof}

\subsection{Curvature estimates}\

In this section, we consider the curvature estimates for equation \eqref{G-eq}.
When $2<k\leq C_{n-1}^{\mathbf{p}-1}+1$ and $0 \leq l<\min\{C_{n-1}^{\mathbf{p}-1}, k\}$,  the following  results can be directly obtained from  Theorem 1.4 and 1.6 in \cite{Zhou-2024}.


\begin{theorem}\label{c2}
Let  $2<k\leq C_{n-1}^{\mathbf{p}-1}+1$, $1\leq\mathbf{p}<n$ and $0 \leq l<\min\{C_{n-1}^{\mathbf{p}-1}, k\}$,  $M\subset \mathbb{R}^{n+1}$ be a closed, star-shaped and  $(\mathbf{p}, k)$-convex hypersurface satisfying the equation \eqref{G-eq} with a positive function $f^{\frac{1}{k-l}}\in C^{2}(\mathbb{S}^{n})$, then
$$|\kappa_i|\leq C, \quad \forall~~~X \in M,~~~1\leq i \leq n.$$
where the constant $C$ depends on $n, k, l, \mathbf{p},q,|\rho|_{C_1},|f^{\frac{1}{k-l}}|_{C_2},\max_{\mathbb{S}^n}f,\min_{\mathbb{S}^n}f$.
\end{theorem}

Next, we consider  the curvature estimates for equation \eqref{G-eqk} with $0<k\leq N= C_{n}^{p}$. Denote the equation as follows,
\begin{equation}\label{G-eq2}
F=\sigma_k(\Lambda)=f(\frac{X}{|X|})|X|^b\langle X,\nu\rangle^q=Tu^q.
\end{equation}
where $T(X)=f(\frac{X}{|X|})|X|^b$ and $u=\langle X,\nu\rangle$.

\begin{lemma}\label{convex}
For any $\alpha > 0$, if $(h_{ij}) \in \mathcal{P}_{\mathbf{p}, k}$, then
\begin{equation*}
(\sigma_k(\Lambda))^{ij,mr} h_{ij;s} h_{mr;s} \leq \sigma_k \left( \frac{(\sigma_k(\Lambda))_s}{\sigma_k(\Lambda)} - \frac{(\sigma_1(\Lambda))_s}{\sigma_1(\Lambda)} \right) \left((\alpha + 1) \frac{(\sigma_k(\Lambda))_s}{\sigma_k(\Lambda)}- (\alpha - 1) \frac{(\sigma_1(\Lambda))_s}{\sigma_1(\Lambda)} \right).
\end{equation*}
\end{lemma}

\begin{proof}
By  Lemma 3.1 in ~\cite{HX-2013}, for any $\alpha > 0$,
\begin{eqnarray*}
\begin{aligned}
&\sum_{I,J,M,R}^{N}\frac{\partial^2 \sigma_k(\Lambda)}{\partial W_{IJ}\partial W_{MR}}W_{IJ;s}W_{MR;s}\\
\leq& \sigma_k(\Lambda) \left( \frac{(\sigma_k(\Lambda))_s}{\sigma_k(\Lambda)} - \frac{(\sigma_1(\Lambda))_s}{\sigma_1(\Lambda)} \right) \left((\alpha + 1) \frac{(\sigma_k(\Lambda))_s}{\sigma_k(\Lambda)}- (\alpha - 1) \frac{(\sigma_1(\Lambda))_s}{\sigma_1(\Lambda)} \right).
\end{aligned}
\end{eqnarray*}
where $\{W_{IJ}\}$ are the entries of the matrix of the linear derivation $\mathcal{D}_H$ associated with $H=\{h_{ij}\}$.
Therefore,
\begin{align*}
\sum_{i,j,m,r}^{n}\frac{\partial^2 \sigma_k(\Lambda)}{\partial h_{ij}\partial h_{mr}}h_{ij;s}h_{mr;s}
=&\sum_{I,J,M,R}^{N}\frac{\partial^2 \sigma_k(\Lambda)}{\partial W_{IJ}\partial W_{MR}}
\sum_{i,j,m,r}^n\frac{\partial W_{IJ}}{\partial h_{ij}}\frac{\partial W_{MR}}{\partial h_{mr}}h_{ij;s}h_{mr;s}\\
=&\sum_{I,J,M,R}^{N}\frac{\partial^2 \sigma_k(\Lambda)}{\partial W_{IJ}\partial W_{MR}}\sum_{i,j}^n\frac{\partial W_{IJ}}{\partial h_{ij}}h_{ij;s}
\sum_{m,r}^n\frac{\partial W_{MR}}{\partial h_{mr}}h_{mr;s}\\
=&\sum_{I,J,M,R}^{N}\frac{\partial^2 \sigma_k(\Lambda)}{\partial W_{IJ}\partial W_{MR}} W_{IJ;s}W_{MR;s}\\
\leq& ~ \sigma_k \left( \frac{(\sigma_k)_s}{\sigma_k} - \frac{(\sigma_1)_s}{\sigma_1} \right) \left((\alpha + 1) \frac{(\sigma_k)_s}{\sigma_k}- (\alpha - 1) \frac{(\sigma_1)_s}{\sigma_1} \right).
\end{align*}
\end{proof}

\begin{theorem}\label{c22}
Let $1 \leq k\leq n$,   $M\subset \mathbb{R}^{n+1}$ be a closed, star-shaped and  $(\mathbf{p}, k)$-convex hypersurface satisfying the equation \eqref{G-eq2}.
Suppose $f(X)$ is a $C^2$ positive function on $M$ and $q\leq 1$, then
\begin{equation*}\label{eq:3.1}
\sigma_1(H) \leq C.
\end{equation*}
 where the constant $C$ depends on $n, k, q,|\rho|_{C_1},|f|_{C_2},\max_{\mathbb{S}^n}f,\min_{\mathbb{S}^n}f, \min_{\mathbb{S}^n}\rho$.
\end{theorem}

\begin{proof}
Consider the test function $\ln \dfrac{\sigma_1}{u}$.
Then, at its maximum  point $X_0$,
\begin{eqnarray}\label{c331}
\begin{aligned}
0 = \left( \ln \frac{\sigma_1}{u} \right)_i = \frac{(\sigma_1)_i}{\sigma_1} - \frac{u_i}{u},
\end{aligned}
\end{eqnarray}
and
\begin{eqnarray*}\label{c332}
\begin{aligned}
0 \geq F^{ij} \left( \ln \frac{\sigma_1}{u} \right)_{ij} 
=~& \frac{1}{\sigma_1} F^{ij} (\sigma_1)_{ij} - \frac{1}{u} F^{ij} u_{ij}.
\end{aligned}
\end{eqnarray*}
Moreover, we obtain
\begin{equation}\label{c333}
0 \geq \frac{1}{u} F^{ij} (\sigma_1)_{ij} - \frac{\sigma_1}{u^2} F^{ij} u_{ij}.
\end{equation}
From the identity $ u_{ij} = h_{ijk} \langle X, e_k \rangle + h_{ij} - h_{ik} h_{kj}u$,  it follows that
\begin{eqnarray}\label{c334}
\begin{aligned}
- \frac{\sigma_1}{u^2} F^{ij} u_{ij}
= - \frac{\sigma_1}{u^2}(T_su^q+qTu^{q-1}u_s) \langle X, e_k \rangle - \frac{\sigma_1}{u^2} kF + \frac{\sigma_1}{u} F^{ij} h_{ik} h_{kj}.
\end{aligned}
\end{eqnarray}
 Equation \eqref{2rd} yields the identity
$$ h_{ssij} = h_{ijss} -h_{ss} h_{im} h_{mj} + h_{sm} h_{ij} h_{ms}.$$
Consequently,
\begin{align}
\frac{1}{u} F^{ij} (\sigma_1)_{ij}
=&\frac{1}{u} F^{ij} (\sum_{s=1}^nh_{ss})_{ij} \notag \\
=& \frac{1}{u}\sum_{s=1}^n F^{ij} h_{ij;ss} + \frac{kF}{u} |H|^2 - \frac{1}{u} F^{ij}h_{im}h_{mj}\sigma_1 \notag \\
=& \frac{1}{u} \sum_{s=1}^n(Tu^q)_{ss} - \frac{1}{u}\sum_{s=1}^n F^{ij;mr} h_{ij;s} h_{mr;s} + \frac{kF}{u} |H|^2 - \frac{\sigma_1 }{u} F^{ij}h_{im}h_{mj}\notag \\
=& \frac{1}{u} \sum_{s=1}^n\left[T_{ss} u^q + 2q u^{q-1}T_s u_s + q(q-1)Tu^{q-2}u_s^2+qTu^{q-1} u_{ss}\right]\notag \\
&- \frac{1}{u}\sum_{s=1}^n F^{ij;mr} h_{ij;s} h_{mr;s} +  \frac{kF}{u}  |H|^2  - \frac{\sigma_1}{u} F^{ij}h_{im}h_{mj} \notag \\
=& \frac{1}{u} \sum_{s=1}^n(T_{ss} u^q + 2q u^{q-1}T_s u_s)+ q(q-1)Tu^{q-3}\sum_{s=1}^nu_s^2 +qTu^{q-2} [ (\sigma_1)_l(X, e_l) +\sigma_1]\notag \\
& - \frac{1}{u}\sum_{s=1}^n F^{ij;mr} h_{ij;s} h_{mr;s} +  \frac{1}{u}(k-q)F |H|^2 - \frac{\sigma_1 }{u} F^{ij}h_{im}h_{mj}\label{c335}.
\end{align}
where $|H|^2=\sum_{s,m}h_{sm}h_{ms}$.
On the other hand, it follows from \eqref{c331} that
$$
\frac{(\sigma_k)_s}{\sigma_k} = q \frac{u_s}{u} + \frac{T_s}{T},\quad
\frac{(\sigma_1)_s}{\sigma_1} = \frac{u_s}{u}.$$
Together with Lemma \ref{convex}, we derive
\begin{equation}\label{c333555}
F^{ij,mr} h_{ij;s} h_{mr;s} \leq {\sigma_k}\Big[(\alpha+1)\frac{T_s^2}{T^2}+(2\alpha q+2q-2\alpha)\frac{T_s}{T}\frac{u_s}{u}+\Big((\alpha+1)q^2-2\alpha q+(\alpha-1)\Big)\frac{u_s^2}{u^2}\Big].
\end{equation}
The identities \eqref{xi} and \eqref{xij} give the first and second derivatives of  $|X|$,
\begin{align}
|X|_i &= |X|^{-1} \langle X, e_i \rangle, \\
|X|_{ii} &= -|X|^{-3} \langle X, e_i \rangle^2 + |X|^{-1}(1 - h_{ii} u).
\end{align}
Using these,  we calculate the derivatives of $T$,
\begin{equation}\label{c3335551}
T_s = (f |X|^b)_s = f_s |X|^b + b f |X|^{b-2} \langle X, e_s \rangle \geq -C,
\end{equation}
and
\begin{align}\label{c3335552}
T_{ss} =&\; f_{ss} |X|^b
        + 2b f_s |X|^{b-2} \langle X, e_s \rangle
        + b(b-1) f |X|^{b-4} \langle X, e_s \rangle^2 \nonumber \\
       &\; - b f |X|^{b-4} \langle X, e_s \rangle^2
        + b f |X|^{b-2}
        - b f |X|^{b-2} u \sigma_1 \nonumber \\
      \geq&\; -C - C \sigma_1.
\end{align}
where $C$ is a constant depending on $\min_{\mathbb{S}^n}\rho, |\rho|_{C_1},|f|_{C_2},\max_{\mathbb{S}^n}f,\min_{\mathbb{S}^n}f$.

Plugging \eqref{c334} - \eqref{c3335552} into \eqref{c333}, we obtain
\begin{equation*}
\begin{aligned}
0\geq&- \sigma_1T_su^{q-2} \langle X, e_k \rangle - \frac{\sigma_1}{u^2} kF + \frac{1}{u} T_{ss} u^q + q(q-1)Tu^{q-3}u_s^2+q\sigma_1Tu^{q-2}+  \frac{k-q}{u}F |H|^2 \\
&-(\alpha+1)\frac{T_s^2}{T}u^{q-1}-(2\alpha q-2\alpha)T_su^{q-2}u_s-\Big((\alpha+1)q^2-2\alpha q+(\alpha-1)\Big) Tu^{q-3}u_s^2\\
\geq&(1-q)\Big(\alpha(q-1)+1\Big)Tu^{q-3}u_s^2-(2\alpha q-2\alpha)T_su^{q-2}u_s+(k-q)Tu^{q-1}|H|^2 -C\sigma_1-C\\
\geq&(1-q)\Big(\alpha(q-1)+1\Big)Tu^{q-3}u_s^2+\frac{k-q}{2}Tu^{q-1}|H|^2 -C\sigma_1-C.
\end{aligned}
\end{equation*}

One has the $C^2$ estimate if
$$(1-q)\Big(\alpha(q-1)+1\Big) \geq 0,$$
which is satisfied by taking $q \leq 1$ and $\alpha > \frac{1}{1 - q} > 0$.
\end{proof}

\subsection{A Constant Rank Theorem}

To preserve the convexity of solutions when applying the continuity method,  we can utilize a specific version of
 the following Constant Rank Theorem.
\begin{theorem}\label{crt}
Let $M \subset \mathbb{R}^{n+1}$ be a $(\mathbf{p},k)$-convex hypersurface satisfying equation \eqref{G-eq}. Assume that the second fundamental form $\{h_{ij}\}$ is positive semi-definite, and that $f$ is a positive smooth function such that
 $$f^{\frac{1}{k-l}}(\frac{X}{|X|}) |X|^{\frac{b}{k-l}}~ \mbox{is locally concave in}~\mathbb{R}^{n+1}.$$
Then $\{h_{ij}\}$ is positive definite on $M$.
\end{theorem}

\begin{proof}

Suppose $A=\{h_{ij}\}$ attains its minimum rank $m$ at some point $X_0\in M$, we have $\sigma_m(A)(X_0)>0$ and $\sigma_{m+1}(A)(X_0)=0$. Then there exists a small open neighborhood $M_0$ of $X_0$ and a small positive constant $C_0$ such that $\sigma_m(A)(X)\geq C_0>0$. We may assume that $k\leq m\leq n-1$, otherwise we are done.

For convenience, we denote $\kappa=(\kappa_1,\cdots,\kappa_n)$ and $\kappa_i$ are eigenvalues of $A$.
For each $X\in M_0$, choose a local orthonormal frame $\{e_A\}$ in a neighborhood of $X$  such that $\{e_1, e_2, \cdots, e_n\}$ is tangent to
$M_0$, $e_{n+1}$ is the unit normal, and $\{h_{ij}\}$ is diagonal at
$X$ with $h_{11}\leq h_{22}\leq\cdots\leq h_{nn}$.
Consider the test function
$$\phi(X)=\sigma_{m+1}(A).$$
Following the notations in \cite{Gu02}, we say that $h(y)\lesssim k(y)$ provided there exists positive constants $C_1$ and $C_2$ such that
$$(h-k)(y)\leq\left(C_1|\nabla\phi|+C_2\phi\right)(y),$$
and we write $h(y)\sim k(y)$ if $h(y)\lesssim k(y)$ and $k(y)\lesssim h(y)$.

Let $B=\{1,\cdots,n-m\}$ and $G=\{n-m+1,\cdots,n\}$ be the sets of indices for eigenvalues $\kappa_i$. Let $\kappa_B=(\kappa_1, \cdots, \kappa_{n-m})$ be the ``bad" eigenvalues of $A$ and $\kappa_G=(\kappa_{n-m+1}, \cdots, \kappa_n)$ be the ``good" eigenvalues of $A$, for convenience, we also write $B=\kappa_B, G=\kappa_G$ if there is no confusion. Hence we get
\begin{equation}\label{equ1}
  0\sim\phi(X)\sim\sigma_{m+1}(A)\sim\sigma_m(G)\sum_{i\in B}h_{ii}\sim\sum_{i\in B}h_{ii}.
\end{equation}
and
\begin{equation}\label{equ11}
  0\sim\phi_{\alpha}\sim\sigma_m(G)\sum_{i\in B}h_{ii\alpha}\sim\sum_{i\in B}h_{ii\alpha}.
\end{equation}

For the convenience of calculation, equation \eqref{G-eq} can be expressed as
\begin{equation}\label{eq1}
 \left(\frac{\sigma_k}{\sigma_l}(\Lambda(\kappa))\right)^{\frac{1}{k-l}}=\widetilde{f}(X)\langle X, \nu\rangle^{\frac{q}{k-l}}:=\widetilde{\psi},
\end{equation}
where $\widetilde{f}=f^{\frac{1}{k-l}}(\frac{X}{|X|}) |X|^{\frac{b}{k-l}}$.
Denote $\widetilde{F}=\left(\frac{\sigma_k}{\sigma_l}(\Lambda(\kappa))\right)^{\frac{1}{k-l}}$, then
 $$\widetilde{F}^{ij}=\frac{\partial \widetilde{F}}{\partial h_{ij}},\quad
\widetilde{F}^{ij,rs}=\frac{\partial^2 \widetilde{F}}{\partial h_{ij}\partial h_{rs}}.$$
 Differentiating equation \eqref{eq1} twice yields
\begin{eqnarray}\label{eqnar1-1}
\sum_{\alpha,\beta=1}^{n}\widetilde{F}^{\alpha\beta}h_{\alpha\beta i}=\widetilde{\psi}_i,
\quad \quad \sum_{\alpha,\beta,r,s=1}^{n}\widetilde{F}^{\alpha\beta}h_{\alpha\beta ii}+\sum_{\alpha,\beta=1}^{n}\widetilde{F}^{\alpha\beta,rs}h_{\alpha\beta i}h_{rsi}=\widetilde{\psi}_{ii}.
\end{eqnarray}
Proceeding as in \cite[Theorem 4.2, (4.36)]{Gu02}, it follows that
\begin{eqnarray}\label{eqnar1}
\sum_{\alpha=1}^n \widetilde{F}^{\alpha\alpha}\phi_{\alpha\alpha} &\sim&\sigma_m(G)\sum_{\alpha=1}^n\sum_{i\in B}\widetilde{F}^{\alpha\alpha}h_{ii\alpha\alpha} -\sigma_{m-1}(G)\sum_{\alpha=1}^n\sum_{i, j\in B}\widetilde{F}^{\alpha\alpha}h_{ij\alpha}^2\\
 \nonumber &&-2\sum_{\alpha=1}^n\sum_{i\in B, j\in G}\sigma_{m-1}(G|j)\widetilde{F}^{\alpha\alpha}h_{ij\alpha}^2.
\end{eqnarray}
Moreover, in view of the Gauss equation and \eqref{eqnar1-1}, for each $i\in B$,
\begin{equation}\label{eqnar2}
\sum_{\alpha=1}^n\widetilde{F}^{\alpha\alpha}h_{ii\alpha\alpha} = \sum_{\alpha=1}^n\widetilde{F}^{\alpha\alpha}(h_{\alpha\alpha ii}+h_{ii}h_{\alpha\alpha}^2-h_{ii}^2h_{\alpha\alpha})
\sim \widetilde{\psi}_{ii}-\sum_{\alpha,\beta,r,s=1}^{n}\widetilde{F}^{\alpha\beta,rs}h_{\alpha\beta i}h_{rsi}.
\end{equation}
Following the proof of (4.26)-(4.27) in \cite{Gu02},  we deduce from \eqref{eqnar1} and \eqref{eqnar2} that
\begin{equation*}\label{eqnar3}
\frac{1}{\sigma_m(G)} \sum_{\alpha=1}^n \widetilde{F}^{\alpha\alpha} \phi_{\alpha\alpha} \sim \sum_{i \in B} \tilde{\psi}_{ii} - \sum_{i \in B} (I_i + II_i + III_i),
\end{equation*}
where
\begin{align*}
I_i &= \sum_{\alpha,\beta \in G} \tilde{F}^{\alpha\alpha} h_{\alpha\alpha i} h_{\beta\beta i} + 2 \sum_{\substack{\alpha,\beta \in G \\ \alpha < \beta}} \frac{\tilde{F}^{\alpha\alpha}- \tilde{F}^{\beta\beta}}{h_{\alpha\alpha} - h_{\beta\beta}} h_{\alpha\beta i}^2 + 2 \sum_{\alpha \in G} \frac{\tilde{F}^{\beta\beta}}{h_{\alpha\alpha}} h_{\alpha\beta i}^2, \\
II_i &= \sum_{\substack{\alpha \in G \\ \beta \in B}} \left[ 2 \tilde{F}^{\alpha\alpha,\beta\beta} h_{\alpha\alpha i} h_{\beta\beta i} + 2 \frac{ \tilde{F}^{\alpha\alpha} - \tilde{F}^{\beta\beta}}{h_{\alpha\alpha} - h_{\beta\beta}} h_{\alpha\beta i}^2 + 2 \frac{\tilde{F}^{\beta\beta}}{h_{\alpha\alpha}} h_{\alpha\beta i}^2 + \big( \sum_{k=n-l+1}^n \frac{1}{h_{kk}} \big) \tilde{F}^{\alpha\alpha} h_{\alpha\beta i}^2 \right], \\
III_i &= \sum_{\alpha,\beta \in B} \tilde{F}^{\alpha\alpha,\beta\beta} h_{\alpha\alpha i} h_{\beta\beta i} + 2 \sum_{\substack{\alpha,\beta \in B \\ \alpha < \beta}} \frac{\tilde{F}^{\alpha\alpha} - \tilde{F}^{\beta\beta}}{h_{\alpha\alpha} -h_{\beta\beta}} h_{\alpha\beta i}^2 + \sum_{\alpha,\beta \in B}  \sum_{k=n-l+1}^n \frac{\tilde{F}^{\alpha\alpha}}{h_{kk}}  h_{\alpha\beta i}^2.
\end{align*}

By proposition \ref{L1}, the function $-\left[\frac{\sigma_{k}}{\sigma_{l}}(\Lambda(W))\right]^{-\frac{1}{k-l}}$ is ``inverse convex" with respect to $W$.
It then follows from Lemma 4.1 of \cite{Gu02} that
\begin{equation*}
\begin{aligned}
I_i + II_i + III_i\gtrsim 0, \quad \forall~~~ i \in B.
\end{aligned}
\end{equation*}
Hence, we get
\begin{eqnarray*}
 \frac{1}{\sigma_m(G)} \sum_{\alpha=1}^n \widetilde{F}^{\alpha\alpha} \phi_{\alpha\alpha}  &\lesssim&  \sum_{i \in B}\widetilde{\psi}_{ii}
\end{eqnarray*}
for any $X\in \mathcal{O}$.
Recalling that
$\widetilde{\psi}=\widetilde{f}(X)\langle X, \nu\rangle^{\frac{q}{k-l}}$,
one deduces from \eqref{equ1} and \eqref{equ11} that, for each \(i\in B\),
\begin{eqnarray*}
  \widetilde{\psi}_{ii} &=& \sum_{A, C=1}^{n+1} \widetilde{f}_{X_A X_C} e_i^A e_i^C\langle X, \nu\rangle^{\frac{q}{k-l}}+\sum_{A=1}^{n+1} \widetilde{f}_{X_A} X_{ii}^A \langle X, \nu\rangle^{\frac{q}{k-l}}\\
  &&+\frac{2q}{k-l}\sum_{A=1}^{n+1} \widetilde{f}_{X_A} e_i^A \langle X, \nu\rangle^{\frac{q}{k-l}-1} h_{ii} \langle X, e_i\rangle+\frac{q}{k-l} (\frac{q}{k-l}-1)\widetilde{f} \langle X, \nu\rangle^{\frac{q}{k-l}-2}h_{ii}^2 \langle X, e_i\rangle^2\\
  &&+ \frac{q}{k-l} \widetilde{f} \langle X, \nu\rangle^{\frac{q}{k-l}-1} \left(\sum_{i=1}^nh_{iij} \langle X, e_j\rangle+h_{ii}-h_{ii}^2 \langle X, e_{n+1}\rangle\right)\\
  &\sim& \sum_{A,C=1}^{n+1} \widetilde{f}_{X_A X_C} e_i^A e_i^C\langle X, e_{n+1}\rangle^{\frac{q}{k-l}}.
\end{eqnarray*}
Note that $\langle X, e_{n+1}\rangle>0$ since $M$ is star-shaped with respect to origin. It implies that
$$\widetilde{F}^{\alpha\beta}\phi_{\alpha\beta} \leq C_1|\nabla\phi|+C_2\phi.$$
Therefore by the strong minimum principle, $\phi\equiv0$ in $\mathcal{O}$, thus $\{x:\phi(x)=0\}$ is an open and closed set. So $\phi\equiv0$, i.e., $\{h_{ij}\}$ is  constant rank on $M$.  Since $M$ is compact, there is a point $X\in M$ such that all the principal curvatures of $M$ at $X$ are positive. The proof is complet.
\end{proof}

\subsection{Existence and uniqueness }\

Let $\kappa = \kappa(\rho) = (\kappa_1(\rho), \ldots, \kappa_n(\rho))$  be the eigenvalues of the second fundamental form $(h_{ij})$ with respect to the first fundamental form $(g_{ij})$ of the spherical graph defined by $\rho$, and $\tilde{F}(\Lambda(\rho)) = \left[\frac{\sigma_k(\Lambda(\rho))}{\sigma_l(\Lambda(\rho))}\right]^{\frac{1}{k-l}}$. Equation~\eqref{G-eq} can be written as
\begin{equation}\label{qceq}
\tilde F(\Lambda(\rho))   = f^{\frac{1}{k-l}} \rho^{\frac{b+2q}{k-l}} (\rho^2 + |\nabla \rho|^2)^{-\frac{q}{2(k-l)}} := K(x, \rho, \nabla \rho).
\end{equation}
Then we denote $\tilde{F}(x, \rho, \nabla\rho, \nabla^2\rho) := \tilde{F}(\Lambda(\rho))$. Let $L$ denote the linearized operator of $\tilde{F}(\Lambda(\rho)) - {K}(x, \rho, \nabla\rho)$ at a  $(\mathbf{p}, k)$-convex  solution $\rho$ of equation \eqref{qceq}, i.e.,
$$
Lv = \sum_{i,j} \frac{\partial \tilde{F}}{\partial \rho_{ij}} v_{ij} + \sum_k \left( \frac{\partial \tilde{F}}{\partial \rho_k} - \frac{\partial K}{\partial \rho_k} \right) v_k + \left( \frac{\partial \tilde{F}}{\partial \rho} - \frac{\partial K}{\partial \rho} \right) v.
$$
The following Lemmas will  be used in the proof of the existence results.
\begin{lemma}\label{linearized}
Let $-b-q-k+l>0$, if $w$ satisfies $Lw = 0$ on $\mathbb{S}^n$, then $w \equiv 0$ on $\mathbb{S}^n$.
\end{lemma}

\begin{proof}

From \eqref{gij} and  \eqref{hij}, we obtain $\lambda(t\rho)=t^{-1}\lambda(\rho)$, and   $\Lambda(t\rho) = t^{-1}\Lambda(\rho)$. Thus
\begin{equation}\label{c61}
\begin{aligned}
\tilde{F}\big(x, t\rho, \nabla(t\rho), \nabla^2(t\rho)\big) = \tilde{F}\big(\Lambda(t\rho)\big) = \tilde{F}\big(t^{-1}\Lambda(\rho)\big)=\frac{1}{t}\tilde{F}\big(\Lambda(\rho)\big).
\end{aligned}
\end{equation}
Applying  $\left.\frac{d}{dt}\right|_{t=1}$ to both sides of \eqref{c61} yields
\begin{eqnarray}\label{c62}
\begin{aligned}
\sum_{i,j} \frac{\partial \tilde{F}}{\partial \rho_{ij}} \rho_{ij} + \sum_k \frac{\partial \tilde{F}}{\partial \rho_k} \rho_k + \frac{\partial \tilde{F}}{\partial \rho} \rho  = -\tilde{F}\big(\Lambda(\rho)\big).
\end{aligned}
\end{eqnarray}
On the other hand,
\begin{equation}\label{c6666}
\begin{aligned}
{K}(x, t\rho, \nabla(t\rho))= f^{\frac{1}{k-l}} (t\rho)^{\frac{b+2q}{k-l}} ((t\rho)^2 + |\nabla (t\rho)|^2)^{-\frac{q}{2(k-l)}}
= t^{\frac{b+q}{k-l}} {K}(x, \rho, \nabla\rho).
\end{aligned}
\end{equation}
We also apply $\left.\frac{d}{dt}\right|_{t=1}$ to this equation,
\begin{equation}\label{c63}
\begin{aligned}
\sum_k \frac{\partial {K}}{\partial \rho_k} \rho_k + \frac{\partial {K}}{\partial \rho} \rho =\frac{b+q}{k-l} {K}(x, \rho, \nabla\rho).
\end{aligned}
\end{equation}
Combining \eqref{c62}, \eqref{c63} and $\frac{b+q}{k-l}<-1$, we obtain
$$L\rho = -\tilde{F}\big(\Lambda(\rho)\big) - \frac{b+q}{k-l} {K}(x, \rho, \nabla\rho)= -(\frac{b+q}{k-l}+1) {K}(x, \rho, \nabla\rho) > 0.$$
Set $w = v\rho$, we get
\begin{equation*}\label{c64}
\begin{aligned}
 Lw =& \rho\Big( \sum_{i,j} \frac{\partial \tilde{F}}{\partial \rho_{ij}} v_{ij} + \sum_k \left( \frac{\partial \tilde{F}}{\partial \rho_k} - \frac{\partial K}{\partial \rho_k} \right) v_k \Big)+2\sum_{i,j} \frac{\partial \tilde{F}}{\partial \rho_{ij}}v_i\rho_{j}\\
&+v\Big(\sum_{i,j} \frac{\partial \tilde{F}}{\partial \rho_{ij}} \rho_{ij} + \sum_k \left( \frac{\partial \tilde{F}}{\partial \rho_k} - \frac{\partial K}{\partial \rho_k} \right) \rho_k + \left( \frac{\partial \tilde{F}}{\partial \rho} - \frac{\partial K}{\partial \rho} \right) \rho\Big)\\
= &\rho \Big( \sum_{i,j} \frac{\partial \tilde{F}}{\partial \rho_{ij}} v_{ij} + \sum_k \left( \frac{\partial \tilde{F}}{\partial \rho_k} - \frac{\partial K}{\partial \rho_k} \right) v_k \Big)+2\sum_{i,j} \frac{\partial \tilde{F}}{\partial \rho_{ij}}v_i\rho_{j} + vL\rho.
\end{aligned}
\end{equation*}
Suppose $v$ attains its maximum at point $X_0$, then $\nabla v(X_0)=0$ and $\{v_{ij}(X_0)\}\leq 0$.
By \eqref{hij}, we have $\left( \frac{\partial \tilde{F}}{\partial \rho_{ij}} \right) < 0$.
Therefore, at the maximum point $X_0$,
$$0=Lw\geq v(X_0)L\rho.$$
Since $L\rho > 0$, it implies $ v(X_0)\leq0$. Similarly,  $v_{\min }\geq0$. That is $v \equiv 0$. Hence $w \equiv 0$, the linearized operator $L$ is invertible at any positive solution of equation \eqref{G-eq}.
\end{proof}

\begin{lemma}\label{uni}
Let $-b-q-k+l>0$. Suppose $\rho_1, \rho_2$ are  $(\mathbf{p}, k)$-convex solutions of equation \eqref{G-eq}. Then $\rho_1 \equiv \rho_2$.
\begin{proof}
We prove it by contradiction. Suppose $\rho_2 > \rho_1$ somewhere on $\mathbb{S}^n$.
Then there exists a constant $t \geq 1$  such that
$$t\rho_1 \geq \rho_2 \quad \text{on} \ \mathbb{S}^n, \quad t\rho_1 = \rho_2 \quad \text{at some point} \ x \in \mathbb{S}^n.$$
Using  \eqref{c6666} and the assumption $\frac{b+q}{k-l}<-1$, we obtain
\begin{equation*}
\begin{aligned}
K(x, t\rho_1, \nabla(t\rho_1))
= t^{\frac{b+q}{k-l}}K(x, \rho_1, \nabla\rho_1)
= t^{\frac{b+q}{k-l}}\tilde{F}(\Lambda(\rho_1))
\leq t^{-1}\tilde{F}(\Lambda(\rho_1))
=& \tilde{F}(\Lambda(t\rho_1)).
\end{aligned}
\end{equation*}
It follows that
\begin{equation}\label{uni111}
\begin{aligned}
\tilde{F}(\Lambda(t\rho_1)) - K(x, t\rho_1, \nabla(t\rho_1)) \geq 0, \quad
\tilde{F}(\Lambda(\rho_2)) - K(x, \rho_2, \nabla\rho_2) = 0.
\end{aligned}
\end{equation}
Hence
$${L}(t\rho_1 - \rho_2) \geq 0,$$
where ${L}$ is the linearized  elliptic operator of $\tilde{F}(\Lambda(\rho)) - {K}(x, \rho, \nabla\rho)$. The strong maximum principle yields $t\rho_{1}-\rho_{2}\equiv0$ on $\mathbb{S}^{n}$. Since $\frac{b+q}{k-l}<{-1}$, from \eqref{uni111}, we conclude that $t=1$.
\end{proof}
\end{lemma}

Now we are ready to prove  Theorem~\ref{k-convex} {\em(\romannumeral1)} and Theorem~\ref{k-convexk1} {\em(\romannumeral1)}  for the case $-b-q - k+l>0 $ by the continuity method.


\begin{proof}[\textbf{Proof of Theorem~\ref{k-convex} (\romannumeral1)}]
For $2<k\leq C_{n-1}^{\mathbf{p}-1}+1$, $0 \leq l<\min\{C_{n-1}^{\mathbf{p}-1}, k\}$,  and $\mathbf{p}>1$,
we consider the following family of equations
\begin{equation}\label{ft}
F_t=\frac{\sigma_k(\Lambda)}{\sigma_l(\Lambda)}=f_t(X)|X|^b\langle X,\nu\rangle^q=f_t \rho^{b+2q} (\rho^2 + |\nabla \rho|^2)^{-q/2} \quad \text{on } \mathbb{S}^n,
\end{equation}
where $f_t= [tf^{\frac{1}{k-l}}+(1-t)(\frac{C_N^k}{C_N^l}\mathbf{p}^{k-l})^{\frac{1}{k-l}}]^{k-l}$ with $0 \leq t \leq 1$. Denote
$$I = \{t \in [0,1] |~equation~\eqref{ft}~has~a~(\mathbf{p}, k)-convex~ solution\}.$$
$I$ is nonempty because $\rho = 1$ is a solution for $t = 0$. By Lemma~\ref{linearized}, we know the linearized operator is invertible.
Therefore, by the implicit function theorem, for each $t_0 \in I$, there exists a neighborhood $\mathcal{N}$ of $t_0$  such that for every
$t \in \mathcal{N}$, equation \eqref{ft} admits a $(\mathbf{p}, k)$-convex~ solution $\rho_t$.
Hence  $I$ is open.
By  the a priori estimates established in Theorem~\ref{c0}, Theorem~\ref{c1-877} and  Theorem~\ref{c2}, together with Evans-Krylov theorem, we obtain
$$
\|\rho_t\|_{C^{3,\alpha}(\mathbb{S}^n)} \leq C,
$$
where $C$ depends only on $n$, $k$, $l$, $\mathbf{p}$, $b, q$, $\min_{\mathbb{S}^n} f$, $\max_{\mathbb{S}^n} f$, $|f|_{C^1}$, $|f|_{C^2}$.
These a priori estimates imply that $I$ is closed.
 Therefore, we conclude that $I=[0,1]$, which means that equation \eqref{G-eq} has a $(\mathbf{p}, k)$-convex~ solution.
The uniqueness part of the theorem follows from Lemma~\ref{uni}.

Next, we prove the strict convexity of the solution under condition \eqref{908998} and the assumption $q\geq 0$.
Suppose $\rho$ is a $(\mathbf{p}, k)$-convex solution for equation \eqref{ft} with $t=t_0$ and $\{h_{ij}\}$ is semi-positive definite. Since $f^{\frac{1}{k-l}}(\frac{X}{|X|}) |X|^{\frac{b}{k-l}}$ is locally concave, it is easy to verify that $f_t^{\frac{1}{k-l}}(\frac{X}{|X|}) |X|^{\frac{b}{k-l}}$ is locally concave with the condition  $q\geq 0$, then by Constant Rank Theorem \ref{crt}, $\{h_{ij}\}$ must be positive definite.
This implies that the solution hypersurface to \eqref{G-eq} must be strictly convex.
\end{proof}

\begin{proof}[\textbf{Proof of Theorem~\ref{k-convexk1} (\romannumeral1)}]
By Theorem~\ref{c0}, Theorem~\ref{c1-877}, and Theorem~\ref{c22}, one readily obtains the necessary a priori estimates for equation \eqref{G-eqk}. The remaining argument is essentially identical to that of Theorem~\ref{k-convex} {\em(\romannumeral1)}, and hence we omit the details.
\end{proof}


\section{Existence and uniqueness in homogeneous case}

In this section, we consider equation \eqref{G-eq} for the case $-b-q -k+l=0$.  We study the following equation
\begin{equation}
\label{eq:5.1}
F=\frac{\sigma_k}{\sigma_l}(\Lambda(\rho))= f|X|^{b-\varepsilon}\langle X,\nu\rangle^q=f \rho^{b+2q-\varepsilon} (\rho^2 + |\nabla \rho|^2)^{-\frac{q}{2}}, \quad \text{on } \mathbb{S}^n
\end{equation}
for any small $\varepsilon > 0$.

\subsection{The a priori estimates}

Following the proof of Theorem \ref{c0}, we can easily derive $C^0$ estimates for equation \eqref{eq:5.1}.

\begin{theorem}
\label{thm:5.1}
Let $M\subset \mathbb{R}^{n+1}$ be a closed, star-shaped and  $(\mathbf{p}, k)$-convex hypersurface satisfying the equation \eqref{eq:5.1} with a positive function $f \in C^2(\mathbb{S}^n)$.
If $-b-q -k+l=0$, then
$$
\frac{C_N^l \min\limits_{\mathbb{S}^n} f}{C_N^k \mathbf{p}^{k-l}} \leq \rho^{\varepsilon}(x) \leq \frac{C_N^l \max\limits_{ \mathbb{S}^n} f}{C_N^k \mathbf{p}^{k-l}}, \quad \forall x \in \mathbb{S}^n.
$$
\end{theorem}

Based on $C^0$ estimates, we can obtain the following $C^1$ estimates.

\begin{theorem}\label{7.1}
Let $0 \leq l <k\leq n$, $\mathbf{p}>1$, $M\subset \mathbb{R}^{n+1}$ be a closed, star-shaped and  $(\mathbf{p}, k)$-convex hypersurface satisfying the equation \eqref{eq:5.1} with a  positive function $f\in C^{2}(\mathbb{S}^{n})$.
If $-b-q -k+l=0$ and one of the following conditions holds:
\begin{enumerate}
    \item[{\em(\romannumeral1)}] $q \neq 0$;

    \item[{\em(\romannumeral2)}] $q=0$  and $\frac{|\nabla f|}{f}<2(k-l)\sqrt{\frac{\mathbf{p}-1}{\mathbf{p}}}$.
\end{enumerate}
Then
$$\max_{\mathbb{S}^{n}}|\nabla \log\rho|\leq C,$$
where the constant $C$ depends on $n,k,l,\mathbf{p},\min_{\mathbb{S}^{n}}f, \max_{\mathbb{S}^{n}} f$ and $\max_{\mathbb{S}^{n}}\frac{|\nabla f|}{f}$. In particular, we have
\begin{equation}\label{ljc1101}
1\leq\frac{\max_{\mathbb{S}^{n}}\rho}{\min_{\mathbb{S}^{n}}\rho}\leq e^C.
\end{equation}
\end{theorem}

\begin{proof}

Let $u=-\log \rho$, then equation~\eqref{eq:5.1} becomes
\begin{equation*}\label{eq:5.4}
F=\frac{\sigma_{k}}{\sigma_{l}}(\Lambda(W))=fe^{\varepsilon u}\left(1+|\nabla u|^{2}\right)^{\frac{k-l-q}{2}},
\end{equation*}
where $ W$ is the symmetric matrix representing the linear derivation $\mathcal{D}_A$
associated with $A=\{a_{ij}\}$, as defined in  \eqref{propw}.
Consider the test function $Q=|\nabla u|^{2}$ and assume that the point $x_{0}$ is the maximum point of $Q$. Choosing a local coordinate frame field around $x_{0}$ such that
$$u_{1}=|\nabla u|,\quad \text{and}~~\{u_{ij}\}_{2\leq i,j\leq n}~\text{is ~ diagonal}.$$
By an argument similar to that used in the proof of Theorem \ref{c1-877}, we obtain
\begin{equation}\label{equiiur-015}
\begin{aligned}
 0 \geq&|\nabla u|^{2}\sum_{i=2}^{n}{F}^{ii}+\sum_{i=2}^{n}{F}^{ii}u_{ii}^{2}+\frac{f_1}{f}|\nabla u|F+\varepsilon F|\nabla u|^{2}\\
 \geq &|\nabla u|^{2}\sum_{i=2}^{n}{F}^{ii}+\sum_{i=2}^{n}{F}^{ii}u_{ii}^{2}- |\nabla u| \frac{|\nabla f|}{f}F,
\end{aligned}
\end{equation}
Note that at the point $x_0$ , the matrices $\{u_{ij}\}$ and $\{a_{ij}\}$ are diagonal, and  the eigenvalues of matrix $\{a_{ij}\}$ satisfy
$$\kappa_{1}=1,\quad\kappa_{2}=1+u_{22},\cdots,\quad\kappa_{n}=1+u_{nn}.$$

Without loss of generality, we may assume that $\kappa_2\geq \kappa_3\geq \cdots\geq \kappa_n$.
From \eqref{equiiur-001}, \eqref{equiiur-011} and \eqref{equiiur-015}, it follows that
\begin{eqnarray*}\label{c111115}
\begin{aligned}
 0\geq&\frac{k-l}{\mathbf{p}}\left(\frac{C_N^l}{C_N^k}\right)^{\frac{1}{k-l}} F^{\frac{1}{k-l}}
+ (\mathbf{p}-1)(k-l) \left(\frac{C_N^k}{C_N^l}\right)^{\frac{1}{k-l}}(1+|\nabla u|^2) F^{-\frac{1}{k-l}} \\
&- |\nabla u| \frac{|\nabla f|}{f}-2(k-l).
\end{aligned}
\end{eqnarray*}

If $q = 0$, by the Cauchy-Schwarz inequality, we have
\begin{eqnarray*}
\begin{aligned}
 \left(2(k-l)\sqrt{\frac{\mathbf{p}-1}{\mathbf{p}}} -\frac{|\nabla f|}{f} \right)|\nabla u|\leq 2(k-l).
\end{aligned}
\end{eqnarray*}
Hence $|\nabla u| \leq C$.

If $q\neq 0$,  from Theorem \ref{thm:5.1},
\begin{eqnarray*}
\begin{aligned}
\mathbf{p}^{k-l}\frac{C_N^k}{C_N^l}\cdot \frac{ \min_{\mathbb{S}^{n}} f}{\max_{\mathbb{S}^{n}} f} \leq e^{\varepsilon u} f \leq \mathbf{p}^{k-l} \frac{C_N^k}{C_N^l}\cdot \frac{\max _{\mathbb{S}^{n}} f}{\min_{\mathbb{S}^{n}} f}
\end{aligned}
\end{eqnarray*}
for any $x \in \mathbb{S}^n$. Hence
\begin{eqnarray*}
\begin{aligned}
0\geq& (k-l)\left(\frac{ \min_{\mathbb{S}^{n}} f}{\max_{\mathbb{S}^{n}} f}\right)^{\frac{1}{k-l}}(1+|\nabla u|^{2})^{\frac{1}{2}-\frac{q}{2(k-l)}}  - |\nabla u| \frac{|\nabla f|}{f}-2(k-l)\\
&+ \frac{(k-l)(\mathbf{p}-1)}{\mathbf{p}}  \left(\frac{ \min_{\mathbb{S}^{n}} f}{\max_{S^{n}} f}\right)^{\frac{1}{k-l}}(1+|\nabla u|^{2})^{\frac{1}{2}+\frac{q}{2(k-l)}}
\\
\geq&C(n,\mathbf{p},k,l)\left(\frac{ \min_{\mathbb{S}^{n}}f}{\max_{\mathbb{S}^{n}} f}\right)^{\frac{1}{k-l}}\left((1+|\nabla u|^{2})^{\frac{1}{2}-\frac{q}{2(k-l)}}+(1+|\nabla u|^{2})^{\frac{1}{2}+\frac{q}{2(k-l)}}\right)\\
&-|\nabla u|\frac{|\nabla f|}{f}-2(k-l),
\end{aligned}
\end{eqnarray*}
which implies
$$|\nabla u| \leq C,$$
where $C$ depends on $n, k, l, \mathbf{p}, q$, $\min_{\mathbb{S}^{n}} f, \max_{\mathbb{S}^{n}} f$ and $\max_{\mathbb{S}^{n}} \frac{|\nabla f|}{f}$.
As in the proof of Theorem \ref{c1-877}, it implies \eqref{ljc1101} holds.
\end{proof}

\begin{theorem}\label{thm:5.3c1}
Let $2<k\leq C_{n-1}^{\mathbf{p}-1}+1$, $0 \leq l<\min\{C_{n-1}^{\mathbf{p}-1}, k\}$, $-b-q-k+l=0$, and $q, f$ satisfy (\romannumeral1) or (\romannumeral2) in Theorem \ref{7.1}. Let $M\subset \mathbb{R}^{n+1}$ be a closed, star-shaped, and $(\mathbf{p},k)$-convex hypersurface satisfying \eqref{eq:5.1} with a positive function $f\in C^{2}(\mathbb{S}^{n})$. Set
\[
\overline{\rho}:=\frac{\rho}{\min_{\mathbb{S}^n} \rho},
\]
then there exists a positive constant $C''$, depending only on $n, k, l, \mathbf{p}, q, |f|_{C^2}, \max_{\mathbb{S}^n}f$, and $\min_{\mathbb{S}^n}f$, such that
\[
|\nabla^2 \overline{\rho}| \leq C''.
\]
\end{theorem}

\begin{proof}
It follows that $\overline{\rho}$ satisfies
\[
F=\frac{\sigma_k}{\sigma_l}(\Lambda(\overline{\rho}))
= f(\min_{\mathbb{S}^n}\rho)^{-\varepsilon} \overline{\rho}^{\,b+2q-\varepsilon}
\bigl(\overline{\rho}^{\,2}+|\nabla \overline{\rho}|^2\bigr)^{-\frac{q}{2}}
\qquad \text{on } \mathbb{S}^n.
\]
By Theorem~\ref{7.1}, there exist positive constants $C$ and $C'$, depending only on $n, k, l, \mathbf{p}, q$, $\min_{\mathbb{S}^{n}}f$, $\max_{\mathbb{S}^{n}} f$, and $\max_{\mathbb{S}^{n}}\frac{|\nabla f|}{f}$, but independent of $\varepsilon$, such that
\begin{equation}\label{eq:5.9c1}
1 \leq \overline{\rho} \leq \frac{\max_{\mathbb{S}^n} \rho}{\min_{\mathbb{S}^n} \rho} \leq C,
\end{equation}
and
\begin{equation}\label{eq:5.10c1}
|\nabla \overline{\rho}|
= \frac{\rho}{\min_{\mathbb{S}^n} \rho}\cdot \frac{|\nabla \rho|}{\rho}
\leq \frac{\max_{\mathbb{S}^n} \rho}{\min_{\mathbb{S}^n} \rho}\cdot \frac{|\nabla \rho|}{\rho}
\leq C'.
\end{equation}
Combining \eqref{eq:5.9c1}, \eqref{eq:5.10c1}, and Theorem~\ref{c2}, we obtain
\[
|\nabla^2 \overline{\rho}| \leq C'',
\]
where $C''$ depends  on $n, k, l, \mathbf{p}, q, |f|_{C^2}, \max_{\mathbb{S}^n}f, \min_{\mathbb{S}^n}f$, and is independent of $\varepsilon$.
\end{proof}

We rewrite equation \eqref{G-eqk} in the form
\begin{equation}
\label{eq:5.1c22}
F={\sigma_k}(\Lambda(\rho))= f|X|^{b-\varepsilon}\langle X,\nu\rangle^q=f \rho^{b+2q-\varepsilon} (\rho^2 + |\nabla \rho|^2)^{-\frac{q}{2}}, \quad \text{on } \mathbb{S}^n
\end{equation}
As a consequence of Theorem~\ref{c22}, we obtain the following result.

\begin{theorem}\label{thm:5.3c21}
Let $1\leq k\leq N$, $q\leq1$, $-b-q-k=0$, and $q, f$ satisfy (\romannumeral1) or (\romannumeral2) in Theorem \ref{7.1}. Let $M\subset \mathbb{R}^{n+1}$ be a closed, star-shaped, and $(\mathbf{p},k)$-convex hypersurface satisfying \eqref{eq:5.1c22} with a positive function $f\in C^{2}(\mathbb{S}^{n})$. Set
\[
\overline{\rho}:=\frac{\rho}{\min_{\mathbb{S}^n}\rho},
\]
then there exists a positive constant $C''$, depending only on $n, k, \mathbf{p}, q, |f|_{C^2}, \max_{\mathbb{S}^n}f$, and $\min_{\mathbb{S}^n}f$, such that
\[
|\nabla^2 \overline{\rho}| \leq C''.
\]
\end{theorem}

\subsection{Existence and uniqueness }\

\begin{proof}[\textbf{Proof of Theorem~\ref{k-convex} (\romannumeral2)}] We divide the proof  into three steps.\\
\textbf{Step 1:  Existence:} For any small constant $\varepsilon > 0$, the equation~\eqref{eq:5.1} has a unique strictly convex  solution $\rho_\varepsilon$ by applying the method used in the proof of Theorem \ref{k-convex} {\em(\romannumeral1)}.  Denote $\bar{\rho}_\varepsilon = \dfrac{\rho_\varepsilon}{\min_{\mathbb{S}^n} \rho_\varepsilon}$, then $\bar{\rho}_\varepsilon$ satisfies
$$\frac{\sigma_k}{\sigma_l}(\Lambda(\bar{\rho_\varepsilon}))= f(\min_{\mathbb{S}^n}\rho_\varepsilon)^{-\varepsilon} ({\bar{\rho}_\varepsilon})^{b+2q-\varepsilon} (\bar{\rho}_\varepsilon^2 + |\nabla \bar{\rho}_\varepsilon|^2)^{-\frac{q}{2}}, \quad \text{on } \mathbb{S}^n.$$
Letting $\varepsilon \to 0^+$, we have $|\nabla (\bar{\rho}_\varepsilon)^\varepsilon| = \varepsilon (\bar{\rho}_\varepsilon)^{\varepsilon - 1} |\nabla \bar{\rho}_\varepsilon| \to 0$ by \eqref{eq:5.9c1} and \eqref{eq:5.10c1}. Then $(\min_{\mathbb{S}^n} {\rho}_\varepsilon)^{-\varepsilon}$ converges to a positive constant $\gamma$. Thus $\bar{\rho}_\varepsilon$ converges to a $(\mathbf{p}, k)$-convex solution of equation~\eqref{eq:1.1}. The Constant Rank Theorem (Theorem~\ref{crt})  can maintain the convexity  if  $q\geq 0$ and
$f^{\frac{1}{k-l}}(\frac{X}{|X|}) |X|^{\frac{b}{k-l}}$ is locally concave in $\mathbb{R}^{n+1} \backslash \{0\}$.\\
\textbf{Step 2: Uniqueness of the  solution:} Suppose there are two  strictly convex solutions $\rho_1$ and $\rho_2$ such that
$$\frac{\sigma_k}{\sigma_l}(W(\rho_i))=\gamma f \rho_i^{b+2q} (\rho_i^2 + |\nabla \rho_i|^2)^{-\frac{q}{2}},  \quad i=1,2,$$
where  $\{W_{IJ}\}$ are the entries of the matrix of the linear operator $\mathcal{D}_A$ generated by $A=\{a_{ij}(\rho)\}$  with
$$a_{ij}(\rho)=\frac{1}{ \sqrt{\rho^2+|\nabla \rho|^2}}\Big(\delta^{i}_{j}+\big[-\sigma^{ik}+\frac{\nabla^i\rho
\nabla^k\rho}{\rho^2v^2}\big]\nabla_j\nabla_k(\log \rho)\Big).$$

Denote $F(W(\rho))=\frac{\sigma_k}{\sigma_l}(W(\rho))$. Let $M(\rho) := \frac{F(W(\rho))}{\rho^{b+2q} (\rho^2 + |\nabla \rho|^2)^{-\frac{q}{2}}},$
then $M(\rho_1) - M(\rho_2) = \gamma f - \gamma f = 0$. Since $M$ is invariant under scaling, we may assume $\rho_1 \leq \rho_2$ and $\rho_1(x_0) = \rho_2(x_0)$ for some point $x_0 \in \mathbb{S}^n$. Denote $\rho_t = t\rho_1 + (1-t)\rho_2$ for $0 \leq t \leq 1$, then
\begin{equation*}
\begin{aligned}
0 = M(\rho_2) - M(\rho_1) &= \int_0^1 -\frac{d}{dt} M(\rho_t) dt \\
&= \sum_{i,j} \widetilde{a}_{ij}(x)(\rho_1-\rho_2)_{ij} + \sum_i b_i(x)(\rho_1-\rho_2)_i + c(x)(\rho_1-\rho_2),
\end{aligned}
\end{equation*}
where
\begin{equation*}
\begin{aligned}
\widetilde{a}_{ij} = \int_0^1 \rho_t^{-b-2q-1}(\rho_t^2 + |\nabla \rho_t|^2)^{\frac{q-1}{2}}\frac{\partial F}{\partial a_{si}}[\sigma^{sj}-\frac{D^s\rho_t
D^j\rho_t}{\rho_t^2v^2}] \, dt
 > 0,
\end{aligned}
\end{equation*}
\begin{equation*}
\begin{aligned}
b_i = \int_0^1 - \rho_t^{-b-2q}(\rho_t^2 + |\nabla \rho_t|^2)^{\frac{q}{2}}\left[\frac{\partial F}{\partial a_{sm}}\frac{\partial a_{sm}}{\partial(\rho_t)_i}+\frac{q(\rho_t)_iF}{(\rho_t^2 + |\nabla \rho_t|^2)}\right]
 \, dt,
\end{aligned}
\end{equation*}
\begin{equation*}
\begin{aligned}
c = \int_0^1  -\rho_t^{-b-2q}(\rho_t^2 + |\nabla \rho_t|^2)^{\frac{q}{2}}
\left[\sum_{s,m}\frac{\partial F}{\partial a_{sm}}\frac{\partial a_{sm}}{\partial\rho_t}
+\frac{q\rho_tF}{(\rho_t^2 + |\nabla \rho_t|^2)}
-\frac{(b+2q)F}{\rho_t}\right]
dt.
\end{aligned}
\end{equation*}
Therefore, by the maximum principle, we have $\rho_1-\rho_2 \equiv 0$ on $\mathbb{S}^n$.\\
\textbf{Step 3: Uniqueness of the constant $\gamma$:} Assume that there exists two positive constants $\gamma, \tilde{\gamma}$ and two strictly convex solutions $\rho, \tilde{\rho}$ such that
$$\frac{\sigma_k}{\sigma_l}(W(\rho))=\gamma f \rho^{b+2q} (\rho^2 + |\nabla \rho|^2)^{-\frac{q}{2}},\quad
\frac{\sigma_k}{\sigma_l}(W(\tilde{\rho}))=\tilde{\gamma} f \tilde{\rho}^{b+2q} (\tilde{\rho}^2 + |\nabla \tilde{\rho}|^2)^{-\frac{q}{2}}.$$
Suppose $G = \frac{\rho}{\tilde{\rho}}$ attains its maximum at $x_0 \in \mathbb{S}^n$. Then at $x_0$,
$$
 0= \nabla \log G = \frac{\nabla \rho}{\rho} - \frac{\nabla \tilde{\rho}}{\tilde{\rho}},
$$
and
\begin{eqnarray*}
\begin{aligned}
0\geq\nabla^2 \log G
= \frac{\nabla^2 \rho}{\rho} - \frac{\nabla^2 \tilde{\rho}}{\tilde{\rho}},
\end{aligned}
\end{eqnarray*}
which implies $\rho W(\rho) \geq \tilde{\rho}W(\tilde{\rho})$. Thus at $x_0$ we derive
\begin{eqnarray*}
\begin{aligned}
\frac{\gamma}{{\tilde{\gamma}}} = \frac{\gamma f(x_0)}{\tilde{\gamma} f(x_0)}
&= \frac{\rho^{-b-2q} (\rho^2 + |\nabla \rho|^2)^{\frac{q}{2}} \frac{\sigma_k}{\sigma_l}(W(\rho))}{\tilde{\rho}^{-b-2q} (\tilde{\rho}^2 + |\nabla \tilde{\rho}|^2)^{\frac{q}{2}} \frac{\sigma_k}{\sigma_l}(W(\tilde{\rho}))} \\
&= \frac{\frac{\sigma_k}{\sigma_l}(\rho W(\rho))}{\frac{\sigma_k}{\sigma_l}(\tilde{\rho}W(\tilde{\rho}))} \geq 1.
\end{aligned}
\end{eqnarray*}
Similarly, we can also get $\gamma \leq \tilde{{\gamma}}$ at the minimum point of $G$, hence $\gamma \equiv \tilde{{\gamma}}$.
\end{proof}

\begin{proof}[\textbf{Proof of Theorem~\ref{k-convexk1} (\romannumeral2)}]
The proof is analogous to that in  Theorem~\ref{k-convex} {\em(\romannumeral2)} and so is omitted. It is important to note that Theorem \ref{thm:5.3c21} is applied to establish the convergence of
$\bar{\rho}_\varepsilon$  to a $(\mathbf{p}, k)$-convex solution of equation~\eqref{eq:1.2}.
\end{proof}

\textbf{Conflict of interest statement:}
On behalf of all authors, the corresponding author states that there is no conflict of interest.

\textbf{Data availability statement:}
No datasets were generated or analysed during the current study.

\end{document}